\documentclass[]{article}
\usepackage{geometry}
\usepackage{authblk} 

\usepackage{amsmath}
\usepackage{amssymb}

\usepackage{bigfoot}

\usepackage{enumerate}

\usepackage{fancyhdr}
\usepackage[hidelinks]{hyperref} 

\usepackage{amsthm} 

\newtheorem{proposition}{Proposition}
\newtheorem{lemma}{Lemma}
\newtheorem{claim}{Claim}
\newtheorem{theorem}{Theorem}
\newtheorem*{theorem*}{Theorem}

\newtheorem{example}{Example}

\newtheorem*{example*}{Example}

\newenvironment{subproof}[1][\proofname]{%
  \begin{proof}[#1]%
}{%
  \end{proof}%
}

\newtheorem{innercustomgeneric}{\customgenericname}
\providecommand{\customgenericname}{}
\newcommand{\newcustomtheorem}[2]{%
	\newenvironment{#1}[1]
	{%
		\renewcommand\customgenericname{#2}%
		\renewcommand\theinnercustomgeneric{##1}%
		\innercustomgeneric
	}
	{\endinnercustomgeneric}
}

\newcustomtheorem{customthm}{Theorem}
\newcustomtheorem{customlemma}{Lemma}

\usepackage{mathrsfs}

\usepackage{xcolor} 

\usepackage{caption}
\usepackage{subcaption}
\usepackage{graphicx}

\usepackage{algorithm}
\usepackage{algpseudocode} 

\usepackage{bbm}

\usepackage{bm}

\usepackage{mathtools}

\usepackage[toc,page]{appendix}

\usepackage{caption}
\usepackage{subcaption}

\usepackage{tikz}
\usetikzlibrary{fit,shapes,arrows,positioning}

\tikzset{decision/.style={diamond, draw, text width=4.5em, text badly centered, inner sep=0pt}}
\tikzset{inout/.style={ellipse, draw, text width=7em, text centered, rounded corners,
 minimum width=3.5cm}}
\tikzset{block/.style={rectangle, draw, text width=12em, text centered, rounded corners,
 minimum width=3.5cm}}
 \tikzset{block1/.style={rectangle, draw,fill=gray!20, text width=10em, text centered, rounded corners,
  minimum width=3.5cm}}
\tikzset{line/.style={draw, -latex}}

\title{An accelerated Levin--Clenshaw--Curtis method for the evaluation of highly oscillatory integrals}
\author[1]{Arieh Iserles}
\author[2]{Georg Maierhofer}
\affil[1]{Department of Applied Mathematics and Theoretical Physics, University of Cambridge}
\affil[2]{Mathematical Institute, University of Oxford\vspace{-0.2cm}}

\makeatletter
\def\@maketitle{%
	\newpage
	\null
	\vskip 2em%
	\begin{center}%
		\let \footnote \thanks
		{\Large\bfseries \@title \par}%
		\vskip 1.5em%
		{\normalsize
			\lineskip .5em%
			\begin{tabular}[t]{c}%
				\@author
			\end{tabular}\par}%
		\vskip 1em%
		{\normalsize \@date}%
	\end{center}%
	\par
	\vskip 1.5em}
\makeatother

\date{\today}
\newcommand{\dd}{\mathrm{d}}
\newcommand{\e}{\mathrm{e}}
\newcommand{\ii}{\mathrm{i}}
\newcommand{\ChebyshevT}{\mathrm{T}}

\makeatletter
\newcommand{\doublewidetilde}[1]{{%
  \mathpalette\double@widetilde{#1}%
}}
\newcommand{\double@widetilde}[2]{%
  \sbox\z@{$\m@th#1\widetilde{#2}$}%
  \ht\z@=.9\ht\z@
  \widetilde{\box\z@}%
}
\makeatother

\newcommand\revisionsA[1]{{#1}}
\newcommand\revisionsB[1]{{#1}}
\begin{document}

\pagenumbering{arabic}
\maketitle\ \vspace{-1.4cm}\\
\begin{abstract}
The efficient approximation of highly oscillatory integrals plays an important role in a wide range of applications. Whilst traditional quadrature becomes prohibitively expensive in the high-frequency regime, Levin methods provide a way to approximate these integrals in many settings at uniform cost. In this work, we present an accelerated version of Levin methods that can be applied to a wide range of physically important oscillatory integrals, by exploiting the banded action of certain differential operators on a Chebyshev polynomial basis. Our proposed version of the Levin method can be computed essentially in the same cost as a Fast Fourier Transform in the quadrature points and the dependence of the cost on a number of additional parameters is made explicit in the manuscript. This presents a significant speed-up over the direct computation of the Levin method in current state-of-the-art. We outline the construction of this accelerated method for a fairly broad class of integrals and support our theoretical description with a number of illustrative numerical examples.

\end{abstract}


\section{Introduction}
Highly oscillatory integrals appear in the modelling of a wide range of high-frequency physical phenomena {\cite{engquist_fokas_hairer_iserles_2009}}. The efficient evaluation of these integrals {is thus} an integral part in the numerical simulation of such phenomena, from hybrid-numerical asymptotic methods for wave scattering \cite{chandler2012numerical,chandler2012high} to laser matter interaction in quantum mechanics \cite{iserles2019solving}. As a result in the dawn of modern numerical analysis the numerical approximation of such integrals, highly oscillatory quadrature, has received increasing attention in the literature and a variety of different approaches to this problem have been devised \cite{deano2017computing}.

Amongst the most successful techniques for the problem of computing highly oscillatory integrals are numerical steepest descent \cite{huybrechs2006evaluation,huybrechs2007construction}, complex Gaussian quadrature \cite{deano2009complex,huybrechs2012superinterpolation,asheimhuybrechs2013,celsus2021kissing} and Filon methods \cite{filon1930iii,iserles2004,iserles2005,iserlesnorsett2004,iserlesnorsett2005efficient,Gao2017a,Gao2017b,melenk2010,Dominguez2011,maierhoferthesis2021,maierhofer2021extended,wu2022filon}. A further set of successful techniques, and the focus of this present work are given by Levin methods. These methods were first introduced by Levin in two seminal papers \cite{levin1982procedures,levin1996fast} for the univariate setting, and later extended by Olver \cite{olver2006} also to the multivariate case. In essence, the central idea of the Levin method is to reformulate the evaluation of a highly-oscillatory integral as the solution of a non-oscillatory ODE and then to approximate the solution of this ODE system by collocation. \revisionsB{We note that around the same time of Levin's original work, a similar approach based on a Chebyshev spectral method for the corresponding ODE was developed independently by Hasegawa \& Torii \cite{hasegawa1987indefinite}.} An interesting extension to higher-dimensional domains using the formulation of Levin's equation as a minimisation problem was recently proposed by Ashton \cite{ashton2023}.

This collocation problem is generally speaking fairly expensive - direct solution costs roughly $\mathcal{O}(\nu^3)$ operations, where $\nu$ is the number of Levin collocation points. Recent work has sought to reduce this cost, notably this was achieved by Liu \cite{liu2012fast} who constructed an improved Levin method based on eigendecomposition in the linear collocation system which is able to reduce the cost to $\mathcal{O}(\nu^2)$. An efficient iterative approach to the problem based on fast discrete cosine transforms and GMRES was taken by \cite{olver2010fast} which, however, does not treat the case of vectorial oscillatory weight functions (cf.\ the case $M>1$ below). \revisionsB{Relatedly, a particularly noteworthy set of work addressing this issue was conducted by Keller \cite{keller2007method,keller1998indefinite,keller2012practical} which is based on the spectral method introduced in \cite{hasegawa1987indefinite} and derives efficient recurrence relations for the Chebyshev coefficients of the solution to Levin's equation. This approach notably is applicable to a wide class of oscillators which satisfy differential equations with polynomial coefficients. In the present work we take a similar approach in the Levin collocation setting and extend this to the case when end-point derivatives are included in Levin's collocation system to ensure a smaller asymptotic error of the method \cite{olver2006}. For this} we consider Clenshaw--Curtis points for the collocation problem and note that a slightly modified differential equation actually leads to banded matrix presentations of all relevant operators in the corresponding Chebyshev polynomial basis. By solving this banded matrix problem directly and exploiting fast discrete cosine transforms this idea allows us to construct a highly efficient version of Levin methods with overall cost of roughly $\mathcal{O}(\nu\log \nu)$ thus presenting a significant speed-up over prior work. We provide a detailed construction of the modified differential equation and the corresponding method which can be applied to integrands of various levels of complexity.

This manuscript is structured as follows. In section~\ref{sec:background_and_notation} we specify the types of integrals under consideration and summarise relevant prior results on the Levin method. This is followed by the construction of our new method in section~\ref{sec:accelerated_LCC_method} beginning with simple special cases and ultimately expressing the method in its full generality applicable to all integrals considered in Levin's original works \cite{levin1982procedures,levin1996fast}. Finally, in section~\ref{sec:numerical_experiments}, we provide several numerical examples which demonstrate the favourable performance of our new method in comparison to prior work. Concluding remarks and an outlook of potential future work are provided in section~\ref{sec:conclusions}.

\section{Background and notation}\label{sec:background_and_notation}
\subsection{The types of integrals under consideration}
In this manuscript we are interested in constructing efficient approximations to integrals of the form
\begin{align}\label{eqn:general_integral_levin}
	I_\omega[\mathbf{f}]:=\int_{-1}^1 \langle \mathbf{f},\mathbf{w}\rangle (x)\,\dd x,
\end{align}
where $\mathbf{f}:[-1,1]\rightarrow\mathbb{C}^M$ is an $M$-vector of $\omega$-independent functions, $\mathbf{w}:[-1,1]\rightarrow\mathbb{C}^M$ is an $M$-vector of $\omega$-dependent linearly independent weight functions that oscillate rapidly when $\omega\gg1$, and $\langle\cdot,\cdot\rangle$ is the usual scalar product of two vectors. We assume that the oscillatory weight functions satisfy a differential equation of the form
\begin{align}\label{eqn:differential_equation_weight_function}
	\mathbf{w}'(x)=\mathsf{G}_\omega(x)\mathbf{w}(x),\quad x\in[-1,1],
\end{align}
where $\mathsf{G}_\omega:[-1,1]:\mathbb{C}^{M\times M}$ is a matrix, whose entries are rational functions in $x$ without poles in $[-1,1]$.

\begin{example}[Exponential oscillator]\label{ex:exponential_oscillator}
	A special case is the integral
	\begin{align}\label{eqn:general1d_integral_levin}
		I_\omega[f]=\int_{-1}^1 f(x) \e^{\ii\omega g(x)}\,\dd x,
	\end{align}
	where $g(x)$ is a polynomial of finite degree $d\in\mathbb{N}$, with $g'(x)\neq 0,\, \forall x\in[-1,1]$. In this case \eqref{eqn:differential_equation_weight_function} becomes
	\begin{align*}
		w'(x)=\ii\omega g'(x)w(x), \quad \text{for\ }\ w(x)=\e^{\ii\omega g(x)}.
	\end{align*}
\end{example}
\begin{example}[Bessel weight function, see Example 1 in \cite{levin1996fast}]\label{ex:bessel_weight_function}
	Another example in this class of integrals are finite Hankel transforms of the form
	\begin{align*}
		I_\omega[f]:=\int_{-1}^1f(x)\mathrm{J}_\gamma(\omega (x+a))\,\dd x
	\end{align*}
for any $a\in(-\infty,-1)\cup(1,\infty)$, where $\mathrm{J}_\gamma$ is the Bessel function of the first kind with index $\gamma\in\mathbb{R}$. Here we take $\mathbf{w}(x)=(\mathrm{J}_{\gamma}(\omega (x+a)),\mathrm{J}'_\gamma(\omega (x+a)))^\top$
which satisfies the differential equation \eqref{eqn:differential_equation_weight_function} with
\begin{align*}
	\mathsf{G}_\omega=\begin{pmatrix}
		0&\omega\\
		-\omega+\frac{\nu^2}{\omega(x+a)^{2}}&-\frac{1}{x+a}
	\end{pmatrix}.
\end{align*}
\end{example}
\subsection{The Levin method}
Introduced in two seminal papers by Levin \cite{levin1982procedures,levin1996fast}, the Levin method provides a way of approximating the {integral} \eqref{eqn:general_integral_levin} at frequency-independent cost, by relating the computation of $I_\omega[f]$ to the solution of an ordinary differential equation of non-oscillatory nature. As a first step consider a function $\mathbf{u}:[-1,1]\rightarrow\mathbb{C}^m$ which satisfies
\begin{align}\label{eqn:general_Levin_differential_equation}
	\bm{\mathcal{L}}_\omega\mathbf{u}(x):=\mathbf{u}'(x)+\mathsf{G}_\omega^{{\top}}\mathbf{u}(x)=\mathbf{f}(x),\quad x\in[-1,1],
\end{align}
where we introduced the Levin differential operator $\bm{\mathcal{L}}_\omega=\dd/\dd x+\mathsf{G}^\top_\omega$. We can use the vector-valued function $\mathbf{u}$ to write the integrand of $I_\omega[\mathbf{f}]$ as an exact differential
\begin{align*}
	\langle\mathbf{f},\mathbf{w}\rangle \,\dd x=\left(\langle\mathbf{u}',\mathbf{w}\rangle+\langle\mathsf{G}_\omega^{{\top}}\mathbf{u},\mathbf{w}\rangle\right)\dd x
	=\left(\langle\mathbf{u}',\mathbf{w}\rangle+\langle\mathbf{u},\mathsf{G}_\omega\mathbf{w}\rangle\right)\dd x
	=\left(\langle\mathbf{u}',\mathbf{w}\rangle+\langle\mathbf{u},\mathbf{w}'\rangle\right)\dd x
	&=\dd\langle \mathbf{u},\mathbf{w}\rangle,
\end{align*}
and therefore to compute $I_\omega[\mathbf{f}]$ by evaluating $\mathbf{u}$:
\begin{align*}
	I_\omega[\mathbf{f}]=\int_{-1}^1\dd\langle \mathbf{u},\mathbf{w}\rangle=\langle\mathbf{u}(1),\mathbf{w}(1)\rangle-\langle\mathbf{u}(-1),\mathbf{w}(-1)\rangle.
\end{align*}
The Levin method then proceeds by approximating the function $\mathbf{u}$ using a collocation solution of \eqref{eqn:general_Levin_differential_equation}. For this we choose a sequence of interpolation bases $\{{\psi_n^{[m]}}\}_{n=0,\dots, \nu+2s+1},\,m=1,\dots,M$, then write
\begin{align*}
	\mathbf{q}=\left(\sum_{n=0}^{\nu+2s+1}\alpha_n^{{[1]}}\psi_n^{{[1]}},\dots,\sum_{n=0}^{\nu+2s+1}\alpha_n^{{[M]}}\psi_n^{{[M]}}\right)^\top,
\end{align*}
and solve for the coefficients $\alpha_n^{{[m]}},\, n=0,\dots \nu+2s+1,\,m=1,\dots,M,$ by solving the linear system arising from the collocation equations:
\begin{align}\begin{split}\label{eqn:general_levin_collocation_equations}
	\bm{\mathcal{L}}_\omega\mathbf{q}(x_j)&=\mathbf{f}(x_j),\quad j=0,\dots,\nu+1,\\
	\left[\frac{\dd^l}{\dd x^{l}}\bm{\mathcal{L}}_\omega\mathbf{q}\right]_{x=\pm1}&=\left[\frac{\dd^l}{\dd x^{l}}\mathbf{f}\right]_{x=\pm 1},\quad l=1,\dots,s,
	\end{split}
\end{align}
where $s\in\mathbb{N}\cup\{0\}$, and $-1=x_0<x_1<\cdots<x_\nu<x_{\nu+1}=1$ are some collocation points. In fact, the idea of using confluent collocation points (i.e. including derivative values in this collocation system) was first introduced by Olver \cite{olver2006}, and the original version of the Levin method given in \cite{levin1982procedures,levin1996fast} has $s=0$ in \eqref{eqn:general_levin_collocation_equations}.

The important insight of the Levin method is that in solving for a non-oscillatory solution of \eqref{eqn:general_Levin_differential_equation} the numerical method requires significantly fewer degrees of freedom than would be required to approximate the oscillatory integral $I_\omega[f]$ directly. This idea is supported by several rigorous statements concerning the properties of exact and approximate solutions to \eqref{eqn:general_Levin_differential_equation} and for a detailed treatment we refer the reader to \cite[Appendix]{levin1982procedures} and \cite[\S3.3]{deano2017computing}. However we mention the following important observation made by Olver \cite[Theorem 4.1]{olver2006} concerning the asymptotic error of the Levin method. Let us denote the Levin quadrature rule by
\begin{align*}
	\mathcal{Q}_\omega^{L,[\nu,s]}[\mathbf{f}]:=\langle\mathbf{q}(1),\mathbf{w}(1)\rangle-\langle\mathbf{q}(-1),\mathbf{w}(-1)\rangle.
\end{align*}
Then Olver showed the following asymptotic error estimate, which holds in the limit $\omega\rightarrow\infty$:
\begin{theorem}[Thm~4.1 in \cite{olver2006}, see also Thm.~3.5 in \cite{deano2017computing}]\label{thm:levin_asymptotic_error} Suppose we are in the case of integrals described in Example~\ref{ex:exponential_oscillator}, i.e.
	\begin{align*}
			I_\omega[f]=\int_{-1}^1 f(x) e^{i\omega g(x)}\,\dd x,
	\end{align*}
with $g'(x)\neq 0,\, \forall x\in[-1,1]$ and $g^{(j)}(\pm1)\neq0,\,j=2,\dots,s$. Then the Levin method \revisionsA{based on \eqref{eqn:general_levin_collocation_equations}} has an asymptotic error of the form
\begin{align*}
I_\omega[\mathbf{f}]-\mathcal{Q}_\omega^{L,[\nu,s]}[\mathbf{f}]=\mathcal{O}(\omega^{-s-2}),\quad\omega\rightarrow\infty.
\end{align*}
\end{theorem}
In a later paper, Olver \cite{Olver2007} showed a similar result concerning the error of the Levin method when $m>1$:
\begin{theorem}[Thm.~4.1 in \cite{Olver2007}]\label{thm:matrix_asymptotics}
	Assume the following conditions are met:
	\begin{itemize}
		\item $\mathbf{f}=\mathcal{O}(\tilde{\mathbf{f}}),\mathsf{G}=\mathcal{O}(\tilde{\mathsf{G}}),$ and $\mathbf{w}=\mathcal{O}(\tilde{\mathbf{w}})$,
		\item $\mathsf{G}^{-1}=\mathcal{O}(\hat{\mathsf{G}})$ where $\tilde{\mathsf{G}}\hat{\mathsf{G}}=\mathcal{O}(1)$, $\hat{\mathsf{G}}=o(1)$,
		\item the basis $\{\psi_n^{{[m]}}\}_{n=0,\dots, \nu+2s+1},\,m=1,\dots,M,$ is independent of $\omega$, {and}
		\item the basis $\{\psi_n^{{[m]}}\}_{n=0,\dots, \nu+2s+1},\,m=1,\dots,M,$ can interpolate at the given interior points and endpoint derivatives.
	\end{itemize}
Then, for large $\omega$, $\mathcal{Q}_\omega^{L,[\nu,s]}[\mathbf{f}]$ is well-defined and 
\begin{align*}
	I_\omega[\mathbf{f}]-\mathcal{Q}_\omega^{L,[\nu,s]}[\mathbf{f}]=\mathcal{O}(\tilde{\mathbf{f}}^\top\hat{\mathsf{G}}\mathbf{1}\hat{\mathsf{G}}^s\tilde{\mathbf{w}}),\quad\omega\rightarrow\infty,
\end{align*}
where $\mathbf{1}$ is a vector of length $m$ containing ones.
\end{theorem}

\revisionsA{Note in Thm.~\ref{thm:matrix_asymptotics} $\hat{\mathsf{G}}=o(1)$ ensures that the error of $\mathcal{Q}_\omega^{L,[\nu,s]}[\mathbf{f}]$ decays as $\omega\rightarrow\infty$. Thus both of the above theorems imply} that the Levin method provides a quadrature rule that approximates the integral at frequency-independent cost as $\omega\rightarrow\infty$, and is therefore a suitable methodology for approximating integrals of the form \eqref{eqn:general_integral_levin} very efficiently. \revisionsA{Note that this decaying quadrature error as $\omega\rightarrow\infty$ is only achieved when $x_0=-1, x_{\nu+1}=1$ as in the above construction. \label{text:comment_pm1points}}

However, the linear collocation system \eqref{eqn:chebyshev_interpolation_system_general} is, in general, a dense $M(\nu+2s+2)\times M(\nu+2s+2)$-linear system and hence requires $\mathcal{O}(M^3(\nu+2s+2)^3)$ operations to solve directly. This is the point where banded spectral computations become useful as we shall see in the following section.

\section{The accelerated Levin--Clenshaw--Curtis method}\label{sec:accelerated_LCC_method}
The central observation which we exploit in this present work is that the action of certain differential operators on Chebyshev polynomial basis functions can be represented by banded matrices. This observation was recently exploited in the context of Filon moment computation by Maierhofer et al. \cite{maierhofer2021extended}. The central ingredient is given by the following simple identities:

\begin{lemma}[Eqs.~22.7.4 \& 22.8.3 in {\cite{abramowitz1965handbook}}]\label{lem:banded_operators_Chebyshev_polynomials}Let $\ChebyshevT_n$ denote the Chebyshev polynomial of the first kind of order $n$, and let us use the convention $\ChebyshevT_{-n}(x):=\ChebyshevT_n(x)$ for all $n\geq0$, then for any $n\geq \mathbb{Z}$:
	\begin{align*}
		x\ChebyshevT_n(x)&=\frac{1}{2}\ChebyshevT_{n-1}(x)+\frac{1}{2}\ChebyshevT_{n+1}(x),\,\,\text{and}\,\,
		(1-x^2)\ChebyshevT_n'(x)=\frac{n}{2}\ChebyshevT_{n-1}(x)-\frac{n}{2}\ChebyshevT_{n+1}(x).
	\end{align*}
	In particular, the actions of $x,(1-x^2)\dd/\dd x$ on $\left\{\ChebyshevT_n\right\}_{n=0}^\infty$ are both of bandwidth 3.
\end{lemma}

\revisionsB{Note the use of similar types of identities for Chebyshev polynomials was also exploited in Keller's work \cite{keller2007method,keller1998indefinite,keller2012practical} and in earlier work by Lewanowicz \cite{lewanowicz1991new} who essentially constructed efficient methods for the approximation of Chebyshev coefficients of special functions obeying differential equations with polynomial coefficients.}

\subsection{The accelerated algorithm in the simple scalar case, $\mathbf{M=1,s=0}$}\label{sec:initial_case_m-1_s-1}
Before describing our framework of the accelerated Levin method for general integrals of the form \eqref{eqn:general_integral_levin} we begin by outlining the main ideas based on the slightly simpler type of integrals given in Example~\ref{ex:exponential_oscillator}. In this case, the Levin method requires us to solve the following differential equation by collocation
\begin{align}\label{eqn:Levin_differential_equation_with_g}
	\mathcal{L}_\omega u_1(x):=u_1'(x)+i\omega g'(x) u_1(x)=f(x), \quad x\in[-1,1],
\end{align}
where $g(x)$ is a polynomial of degree $d$. Let us, for the purpose of this section, focus on the case $s=0$, i.e. we seek an approximation $q_1(x)=\sum_{j=0}^{\nu+1}\alpha_j\psi_j(x)$ to $u_1$ for some $\nu\in\mathbb{N}$ and some basis functions $\{\psi_j\}_{j=0}^{\nu+1}$, by solving the following linear system:
\begin{align}\label{eqn:chebyshev_interpolation_system_general}
	\sum_{j=0}^{\nu+1}\alpha_j (\psi_j'(c_m)+i\omega g'(c_m)\psi_j(c_m))=f(c_m),\quad m=0,\dots,\nu+1,
\end{align}
where $-1=c_0<c_1<\cdots<c_{\nu+1}=1$ are the collocation points.

Our central observation is that one can use a Chebyshev polynomial basis for the approximate solution of \eqref{eqn:Levin_differential_equation_with_g} together with Lemma~\ref{lem:banded_operators_Chebyshev_polynomials} in order to achieve a significant acceleration in finding the coefficients $\bm{\alpha}$ as defined in \eqref{eqn:chebyshev_interpolation_system_general}. Ultimately, we will see that our method is able to compute the full approximation to the integral $I_\omega[f]$ in just $\mathcal{O}(\nu\log\nu+d^2\nu)$ operations, where we recall that $d$ is the polynomial degree of the phase function, $d=\deg g$. As is explained in the sequel, we require $\nu$ to be even and assume {$\nu\geq d+1$}. The latter assumption is made since we shall rely on the solution of a banded matrix system, with bandwidth $2d+3$, so in the case $d\geq \nu$ the matrix would be dense and yield no significant {speed-up} over the direct solution of \eqref{eqn:general_discrete_interpolation_conditions}. Our \textit{Ansatz} is, therefore,
\begin{align*}
	q_1(x)=\sum_{l=0}^{\nu+1}\alpha_l \ChebyshevT_l(x),
\end{align*}
and we use Clenshaw--Curtis collocation points $c_m=\cos(m\pi/(\nu+1)),\, 0\leq m\leq \nu+1$. Thus, we wish to find the coefficients $\bm{\alpha}$ from the following set of interpolation conditions
\begin{align}\label{eqn:general_discrete_interpolation_conditions}
	\mathsf{A}\bm{\alpha}=\mathbf{f},
\end{align}
where $\mathsf{A}_{mn}=\mathcal{L}_\omega \ChebyshevT_n(c_m), f_m=f(c_m)$, $m,n=0,\dots,\nu+1$, and here we {denote} by $\mathcal{L}_\omega$ the Levin differential operator as defined in \eqref{eqn:Levin_differential_equation_with_g}. As a first observation, we note that once $\bm{\alpha}$ is found, we can evaluate the Levin approximation to $I_\omega[f]$ using just $\mathcal{O}(\nu)$ additions. Indeed we have $\ChebyshevT_n(\pm1)=(\pm 1)^n$ (see \cite[Eq.~(2.3)]{Gao2017a}), and hence
\begin{align*}
	q_1(\pm1)=\sum_{j=0}^{\nu+1}(\pm 1)^{j}\alpha_j.
\end{align*}
To speed up the solution of \eqref{eqn:Levin_differential_equation_with_g} we will show that $\nu$ of the $\nu+2$ degrees of freedom in the equation can be found through the solution of a banded matrix system of bandwidth $2d+3$.

In order to do so we begin by transforming $\mathcal{L}_\omega$ to a form that has banded action on the Chebyshev polynomial basis. We recall Lemma~\ref{lem:banded_operators_Chebyshev_polynomials} and thus consider the differential operator
\begin{align*}
	\widetilde{\mathcal{L}}_\omega=(1-x^2)\mathcal{L}_\omega=(1-x^2)\frac{\dd}{\dd x}+i\omega (1-x^2)g'(x).
\end{align*}
The effect of replacing $\mathcal{L}_\omega$ by $\widetilde{\mathcal{L}}_\omega$ in the collocation equations \eqref{eqn:general_discrete_interpolation_conditions} is simply to multiply every row of the linear system by a factor of $(1-c_m^2)$, i.e. $\widetilde{\mathsf{A}}_{mn}=(1-c_m^2)\mathsf{A}_{mn},\, m,n=0,\dots,\nu+1$, where $\widetilde{\mathsf{A}}_{mn}=\widetilde{\mathcal{L}}_\omega \ChebyshevT_n(c_m)$, and $\tilde{f}_m=(1-c_m^2)f_m, m=0,\dots,\nu+1$. Thus the solution $\bm{\alpha}$ of  \eqref{eqn:general_discrete_interpolation_conditions} must also satisfy the $(\nu+2)\times (\nu+2)$ matrix system
\begin{align}\label{eqn:modified_discrete_interpolation_conditions}
	\widetilde{\mathsf{A}}\bm{\alpha}=\tilde{\mathbf{f}}.
\end{align}
By construction (since $c_0=1,c_{\nu+1}=-1$) we have $\tilde{f}_{0}=\tilde{f}_{\nu+1}=0$ and $\widetilde{\mathsf{A}}_{0n}=\widetilde{\mathsf{A}}_{(\nu+1)n}=0$ for all $0\leq n\leq \nu+1$. Hence the linear system \eqref{eqn:modified_discrete_interpolation_conditions} is ill-posed. We can nevertheless work with the system in the following way: Let $\mathsf{P}_{\nu}$ be projection onto the middle $\nu$ coordinates, i.e. $\mathsf{P}_{\nu}:(x_0,\dots,x_{\nu+1})\mapsto(0,x_1,\dots,x_{\nu},0)$. \revisionsA{By slight abuse of notation we will in the following sometimes use $\mathsf{P}_{\nu}\mathbf{x}$ to denote $(x_1,\dots,x_{\nu})$ when the dimensionality of the vector is clear from context.} Consider the linear system
\begin{align}\label{eqn:middle_system_Levin_discrete}
	\mathsf{P}_{\nu} \widetilde{\mathsf{A}}\mathsf{P}_{\nu}\bm{\alpha}_0=\mathsf{P}_{\nu} \tilde{\mathbf{f}}. 
\end{align}
This \revisionsA{is equivalent (through multiplication} of each row by a nonzero constant) to the interpolation conditions
\begin{align}\label{eqn:middle_system_Levin_discrete_interpolation_form}
	\sum_{n=1}^\nu (\bm{\alpha}_0)_n\mathcal{L}_\omega\ChebyshevT_n(c_m)=f(c_m),\quad m=1,\dots, \nu.
\end{align}
Our central observation is that we can solve \eqref{eqn:middle_system_Levin_discrete} very efficiently:
\begin{proposition}\label{prop:efficient_inversion_of_subsystem}
	For $\omega$ sufficiently large, the system \eqref{eqn:middle_system_Levin_discrete} has a unique solution $\bm{\alpha}_0\in \{\mathbf{x}\in\mathbb{C}^{\nu+2}\,\big|\,x_0=x_{\nu+1}=0\}$. Moreover, whenever such a solution exists, we can solve \eqref{eqn:middle_system_Levin_discrete} using $\mathcal{O}(\nu\log\nu+d^2\nu)$ operations, by the application of one discrete cosine transform (DCT-I) and the solution of a banded matrix system of bandwidth $2d+3$.
\end{proposition}
\begin{proof}
	We recall $\widetilde{\mathsf{A}}_{mn}=\widetilde{\mathcal{L}}_\omega \ChebyshevT_n(c_m)$. By Lemma~\ref{lem:banded_operators_Chebyshev_polynomials} the action of $\widetilde{\mathcal{L}}_\omega$ on $\{\ChebyshevT_n\}_{n=0}^\infty$ is represented by a banded (infinite) matrix $\mathsf{B}$ with bandwidth $2d+3$ such that
	\begin{align}\label{eqn6:banded_matrix_representation_tildeL}
		\widetilde{\mathcal{L}}_\omega \ChebyshevT_n(x)=\sum_{k=\max\{0,n-(d+1)\}}^{n+d+1}\mathsf{B}_{kn}\ChebyshevT_k(x).
	\end{align}
	Therefore we can express the matrix $\widetilde{\mathsf{A}}$ in the form
	\begin{align*}
		\widetilde{\mathsf{A}}_{mn}=\sum_{k=\max\{0,n-(d+1)\}}^{n+d+1}\mathsf{B}_{kn} \ChebyshevT_k(c_m),\quad m,n=0,\dots, \nu+1.
	\end{align*}
	To simplify this further note that for $0\leq l,m\leq \nu+1$ we have
	\begin{align}\label{eqn:T_nu+1+l_equals_T_nu+1-l}
		\ChebyshevT_{\nu+1+l}(c_m)&=\cos\left(\frac{(\nu+1+l)m}{\nu+1}\pi\right)=(-1)^m\cos\left(\frac{lm}{\nu+1}\pi\right)\\
		&=(-1)^m\cos\left(\frac{(-l)m}{\nu+1}\pi\right)=\ChebyshevT_{\nu+1-l}(c_m).
	\end{align}
	Thus let us define $\widetilde{\mathsf{B}}$ as a $(\nu+2)\times(\nu+2)$ matrix by
	\begin{align}\label{eqn:def_of_tilde{B}_1d_case}
		\widetilde{\mathsf{B}}_{nm}=\begin{cases}
			\mathsf{B}_{nm},&0\leq n<\nu-d, n=\nu+1,\\
			\mathsf{B}_{nm}+\mathsf{B}_{(2\nu+2-n)m},&\nu-d\leq n\leq \nu.
		\end{cases}
	\end{align}
	Then we have, since $\nu>d$,
	\begin{align*}
		\widetilde{\mathsf{A}}_{mn}=\sum_{k=\max\{0,n-(d+1)\}}^{n}\mathsf{C}_{mk}\widetilde{\mathsf{B}}_{kn},
	\end{align*}
	where $\mathsf{C}_{mk}=\ChebyshevT_{k}(c_m)$. We now observe that the action of $\mathsf{C}$ essentially represents a discrete cosine transform (DCT-I): Let $\mathbf{x}\in\mathbb{C}^{\nu+2}$ then
	\begin{align*}
		\left(\mathsf{C}\mathbf{x}\right)_m=\sum_{n=0}^{\nu+1}\mathsf{C}_{mn}x_n=\sum_{n=0}^{\nu+1}\cos\left(\frac{mn\pi}{\nu+1}\right)x_n=\sideset{}{''}\sum_{n=0}^{\nu+1}\cos\left(\frac{mn\pi}{\nu+1}\right)\tilde{x}_n=:\left(\mathcal{C}_{\nu+1}\tilde{\mathbf{x}}\right),
	\end{align*}
	where we denoted by $\mathcal{C}_{\nu+1}$ the DCT-I on the space $\mathbb{C}^{\nu+2}$. Here $\sideset{}{''}\sum_{n=0}^{\nu+1}$ means that for $j=0$ and $j=\nu+1$ the terms in the sum are halved, and we took
	\begin{align*}
		\tilde{x}_n=\begin{cases}
			x_n,& 1\leq n\leq \nu,\\
			2x_n,& n=0,\nu+1.
		\end{cases}
	\end{align*}
	The inverse of $\mathcal{C}_{\nu+1}$ is again a DCT-I, $\mathcal{C}_{\nu+1}^{-1}=2/(\nu+1)\mathcal{C}_{\nu+1}$. It is well-known that the application of $\mathcal{C}_{\nu+1}$ can be computed efficiently in $\mathcal{O}(\nu\log\nu)$ operations \cite{trefethen2008gauss}, so we see that the action of $\mathsf{C}$ and $\mathsf{C}^{-1}$ can both be computed in $\mathcal{O}(\nu\log\nu)$ operations.
	
	Let us now prove the first part of the Lemma: First of all, \revisionsA{\eqref{eqn:middle_system_Levin_discrete} is equivalent to \eqref{eqn:middle_system_Levin_discrete_interpolation_form} which can be written in the form
	\begin{align*}
		\sum_{n=1}^\nu \left(\sum_{k=1}^{\nu}\left(\mathsf{I}+\frac{1}{i\omega }\mathsf{F}\widetilde{\mathsf{D}}\right)_{nk}\ChebyshevT_k(c_m)\right)(\bm{\alpha}_0)_n=\frac{f(c_m)}{i\omega g'(c_m)},\quad m=1,\dots, \nu.
	\end{align*}
		where $\mathsf{F}=\mathrm{diag}(1/g(x_1),\dots,1/g(x_\nu))$, and $\widetilde{\mathsf{D}}$ corresponds to the differentiation matrix of the Chebychev polynomials. For $\omega$ sufficiently large, the Neumann series of $(i\omega)^{-1}\mathsf{F}\widetilde{\mathsf{D}}$ converges and thus $\left(\mathsf{I}+\frac{1}{i\omega }\mathsf{F}\widetilde{\mathsf{D}}\right)$ is invertible. Therefore the system \eqref{eqn:middle_system_Levin_discrete} is equivalent to}
	\begin{align*}
		\sum_{n=1}^{\nu}\alpha_n \ChebyshevT_n(c_m)=\frac{f(c_m)}{i\omega g'(c_m)}+\mathcal{O}(\omega^{-2}),\quad m=0,\dots,\nu+1, \quad\revisionsA{\text{as\ }\omega\rightarrow\infty,}
	\end{align*}
	i.e. $\mathsf{P}_{\nu}\mathsf{C}\mathsf{P}_{\nu}\bm{\alpha}_0=\mathsf{P}_{\nu}\mathbf{h}+\mathcal{O}(\omega^{-2})$. Thus to show that \eqref{eqn:middle_system_Levin_discrete} has a unique solution $\bm{\alpha}_0\in \{\mathbf{x}\in\mathbb{C}^{\nu+2}\,\big|\,x_0=x_{\nu+1}=0\}$ \revisionsA{for sufficiently large $\omega$}, it is sufficient to prove that the equation 
	\begin{align}\label{eqn:sufficient_condition_to_invert_middle_DCTI}
		\mathsf{P}_{\nu}\mathsf{C}\mathbf{y}=\mathsf{P}_{\nu}\mathbf{z},
	\end{align}
	has a solution $\mathbf{y}\in \{\mathbf{x}\in\mathbb{C}^{\nu+2}\,\big|\,x_0=x_{\nu+1}=0\}$ for all $\mathbf{z}\in\mathbb{C}^{\nu+2}$. For this we will exploit the following convenient property of the matrix $\mathsf{C}$:
	\begin{claim}\label{claim6:special_properties_of_DCTI_in_this_space}If $\mathbf{x},\mathbf{y}\in \mathbb{C}^{\nu+2}$ are such that $(\mathsf{I}-\mathsf{P}_{\nu})\mathsf{C} \mathbf{x}=(\mathsf{I}-\mathsf{P}_{\nu})\mathsf{C} \mathbf{y}=\mathbf{0}$, where $\mathsf{I}$ is the $(\nu+2)\times(\nu+2)$ identity matrix, then
		\begin{align*}
			\mathsf{P}_{\nu}\mathbf{x}=\mathsf{P}_{\nu}\mathbf{y}\quad\Longleftrightarrow\quad \mathsf{P}_{\nu} \mathsf{C} \mathbf{x}= \mathsf{P}_{\nu}\mathsf{C}\mathbf{y}.
		\end{align*}
	\end{claim}
	\begin{subproof}[Proof of Claim~\ref{claim6:special_properties_of_DCTI_in_this_space}] We know that the leading and the bottom rows of $\mathsf{C}$ take the forms $C_{0,n}=1,C_{\nu+1,n}=(-1)^n,\, n=0,\dots, \nu+1$. Thus if $(\mathsf{I}-\mathsf{P}_{\nu})\mathsf{C}\mathbf{x}=\mathbf{0}$, then
		\begin{align*}
			x_0+x_{\nu+1}&=-\sum_{n=1}^\nu x_n,\quad x_0-x_{\nu+1}=-\sum_{n=1}^\nu (-1)^nx_n,
		\end{align*}
		and so $x_0,x_{\nu+1}$ are uniquely determined by the remaining entries of $\mathbf{x}$. Therefore we have
		\begin{align*}
			\mathsf{P}_{\nu}\mathbf{x}=\mathsf{P}_{\nu}\mathbf{y}\quad\Longrightarrow\quad \mathbf{x}=\mathbf{y}\quad\Longrightarrow\quad \mathsf{P}_{\nu} \mathsf{C} \mathbf{x}= \mathsf{P}_{\nu}\mathsf{C}\mathbf{y}.
		\end{align*}
		The implication in the other direction follows, by recalling that $\mathsf{P}_{\nu}\mathsf{C}\mathbf{x}=\mathbf{0}$, thus:
		\begin{align*}
			\mathsf{P}_{\nu} \mathsf{C} \mathbf{x}= \mathsf{P}_{\nu}\mathsf{C}\mathbf{y}\,\,\,\Longrightarrow\,\,\, \mathsf{C}\mathbf{x}=\mathsf{P}_{\nu}\mathsf{C}\mathbf{x}=\mathsf{P}_{\nu}\mathsf{C}\mathbf{y}=\mathsf{C}\mathbf{y}\,\,\,\Longrightarrow\,\,\,\mathbf{x}=\mathbf{y}\,\,\,\Longrightarrow\,\,\, \mathsf{P}_{\nu}\mathbf{x}=\mathsf{P}_{\nu}\mathbf{y},
		\end{align*}
		where the penultimate implication follows since $\mathsf{C}$ is invertible.
	\end{subproof}
Now, we recall that $\mathsf{C}$ is (up to rescaling of the first and final coordinates) a DCT-I, thus one can easily check that
	\begin{align}\label{eqn:def_of_y_j}
		\mathbf{y}_1&=(\tfrac12,1,1,1,\dots,1,1,\tfrac12),\\
		\mathbf{y}_2&=(\tfrac12,-1,1,-1,\dots,-1,1,-\tfrac12),
	\end{align}
	are such that $\mathsf{P}_{\nu}\mathsf{C}\mathbf{y}_j=\mathbf{0},$ $j=1,2$. The final coordinate of $\mathbf{y}_2$ has negative sign because $\nu$ is even. Therefore, whenever a solution $\mathbf{y}_0\in\mathbb{C}^{\nu+2}$ of \eqref{eqn:sufficient_condition_to_invert_middle_DCTI} exists, we can find a linear combination of $\mathbf{y}_0,\mathbf{y}_1,\mathbf{y}_2$ which is in $\mathbf{z}\in\{\mathbf{x}\in\mathbb{C}^{\nu+2}\,\big|\,x_0=x_{\nu+1}=0\}$ and also solves \eqref{eqn:sufficient_condition_to_invert_middle_DCTI}. By Claim~\ref{claim6:special_properties_of_DCTI_in_this_space}, one such solution is $\mathbf{y}_0=\mathsf{C}^{-1}\mathsf{P}_{\nu} \mathbf{z}$, thus completing the proof of the first part of the lemma.
	
	Now let us show how we can solve \eqref{eqn:middle_system_Levin_discrete} efficiently whenever a solution exists. By Claim~\ref{claim6:special_properties_of_DCTI_in_this_space} the system \eqref{eqn:middle_system_Levin_discrete} is equivalent to 
	\begin{align}\label{eqn6:banded_matrix_system}
		\mathsf{P}_{\nu} \widetilde{\mathsf{B}}\mathsf{P}_{\nu}\bm{\alpha}_0=\mathsf{P}_{\nu}\mathsf{C}^{-1}\tilde{\mathbf{f}}.
	\end{align}
	Now we note $\mathsf{P}_{\nu}\widetilde{\mathsf{B}}\mathsf{P}_{\nu}$ is a banded matrix of bandwidth $2d+3$, so a solution of \eqref{eqn6:banded_matrix_system} can be found, using Gaussian elimination, in $\mathcal{O}(d^2\nu)$ operations. Thus to find $\bm{\alpha}_0$, we need to apply a single DCT-I and then solve \eqref{eqn6:banded_matrix_system}, i.e. we incur an overall cost of $\mathcal{O}(\nu\log\nu+d^2\nu)$ operations.
\end{proof}
Since $\widetilde{A}_{mn}=(1-c_m^2)A_{mn},\, m,n=0,\dots,\nu+1,$ we therefore found $\bm{\alpha}_0$ such that
\begin{align*}
	\mathsf{P}_{\nu}	\mathsf{A}\bm{\alpha}_0=\mathsf{P}_{\nu} \mathbf{f}.
\end{align*}
By construction $\mathrm{Nullity}(\widetilde{\mathsf{A}})=\mathrm{Nullity}(\mathsf{A})+2$, therefore, provided the collocation problem \eqref{eqn:general_discrete_interpolation_conditions} is soluble for all $f$, we know that there are two linearly independent vectors $\mathbf{v}_1,\mathbf{v}_2\in\mathbb{C}^{\nu+2}$ such that
\begin{align}\label{eqn:desired_null_solutions}
	\mathsf{P}_{\nu}\mathsf{A}\mathbf{v}_j=\mathbf{0},\quad j=1,2.
\end{align}
We can construct those as follows: Let $\mathbf{e}_0=(1,0,\dots,0)^\top$ and $\mathbf{e}_{\nu+1}=(0,\dots,0,1)^\top$ then by the same process as in Prop.~\ref{prop:efficient_inversion_of_subsystem} we can find $\tilde{\mathbf{v}}_1,\tilde{\mathbf{v}}_2\in \{\mathbf{x}\in\mathbb{C}^{\nu+2}\,\big|\,x_0=x_{\nu+1}=0\}$ with
\begin{align}\label{eqn:linear_system_for_elements_of_kernel}
	\mathsf{P}_{\nu} \mathsf{A}\mathsf{P}_{\nu}\tilde{\mathbf{v}}_1=-\mathsf{P}_{\nu}\mathsf{A}\mathbf{e}_0, \quad \mathsf{P}_{\nu} \mathsf{A}\mathsf{P}_{\nu}\tilde{\mathbf{v}}_2=-\mathsf{P}_{\nu}\mathsf{A}\mathbf{e}_{\nu+1}.
\end{align}
We note that the first and last columns of $\mathsf{A}$ have at least one nonzero entry in a row with index $1\leq m\leq\nu$, so $\mathsf{P}_{\nu}\mathsf{A}\mathbf{e}_0,\mathsf{P}_{\nu}\mathsf{A}\mathbf{e}_{\nu+1}\neq\mathbf{0}$. The vectors $\mathbf{v}_{1}=\mathbf{e}_0+\tilde{\mathbf{v}}_1,\mathbf{v}_{2}=\mathbf{e}_{\nu+1}+\tilde{\mathbf{v}}_2$ are then clearly linearly independent (since $\mathbf{e}_0^\top\mathbf{v}_2=0\neq\mathbf{e}_0^\top\mathbf{v}_1$), and they satisfy \eqref{eqn:desired_null_solutions}. The computation of the vectors takes again $\mathcal{O}(\nu\log\nu+d^2\nu)$ operations, as per Prop.~\ref{prop:efficient_inversion_of_subsystem}. We can finally find $\bm{\alpha}$, the solution to \eqref{eqn:general_discrete_interpolation_conditions}, by solving the remaining  linear system:
\begin{align}\label{eqn:remaining_2x2_system}
	\delta_1\mathsf{A}\mathbf{v}_1+\delta_2\mathsf{A}\mathbf{v}_2=\mathbf{f}-\mathsf{A}\bm{\alpha}_0.
\end{align}
Recall that by construction of $\mathbf{v}_{1},\mathbf{v}_2,\bm{\alpha}_0,$ the rows $m=1,\dots,\nu$ in \eqref{eqn:remaining_2x2_system} are trivially satisfied ($0=0$), and so this reduces to a $2\times 2$ linear system for the coefficients $\delta_1,\delta_2$:
\begin{align}\label{eqn:final_2x2system}
	\begin{pmatrix}
		(\mathsf{A}\mathbf{v}_1)_0&(\mathsf{A}\mathbf{v}_2)_0\\
		(\mathsf{A}\mathbf{v}_1)_{\nu+1}&(\mathsf{A}\mathbf{v}_2)_{\nu+1}
	\end{pmatrix}
	\begin{pmatrix}
		\delta_1\\
		\delta_2
	\end{pmatrix}=\begin{pmatrix}
		f_0-(\mathsf{A}\bm{\alpha}_0)_0\\
		f_{\nu+1}-(\mathsf{A}\bm{\alpha}_0)_{\nu+1}
	\end{pmatrix}.
\end{align}
The coefficients in this system can be found in $\mathcal{O}(\nu)$ operations (since for {example} $(\mathsf{A}\mathbf{v}_1)_0=\sum_{n=0}^{\nu+1}\mathsf{A}_{0n}(\mathbf{v}_1)_n$). Furthermore, because $\mathsf{A}$ is invertible, \eqref{eqn:remaining_2x2_system} must have a unique solution for $\delta_1,\delta_2$. This completes our construction of $\bm{\alpha}$, and hence the computation of $\mathcal{Q}^{\mathrm{L},[\nu]}_\omega[f]$.
\subsubsection*{{An} algorithm for the efficient construction of the Levin method ($\mathbf{M=1,s=0}$)}
Let us briefly summarise the above algorithm for constructing the Levin method in the present case in $\mathcal{O}(\nu\log\nu +d^2\nu)$ operations:
\begin{enumerate}
	\item Solve \eqref{eqn:middle_system_Levin_discrete} using the method described in Prop.~\ref{prop:efficient_inversion_of_subsystem}, to find $\bm{\alpha}_0\in \{\mathbf{x}\in\mathbb{C}^{\nu+2}\,\big|\,x_0=x_{\nu+1}=0\}$ with
	\begin{align*}
		\mathsf{P}_{\nu}\mathsf{A} \bm{\alpha}_0=\mathsf{P}_{\nu}\mathbf{f}.
	\end{align*}
	\item Solve \eqref{eqn:linear_system_for_elements_of_kernel} using the method described in Prop.~\ref{prop:efficient_inversion_of_subsystem}, to find two linearly independent vectors $\mathbf{v}_1,\mathbf{v}_2$ with \begin{align*}
		\mathsf{P}_{\nu}\mathsf{A}\mathbf{v}_j=\mathbf{0},\quad j=1,2.
	\end{align*}
	\item Compute the coefficients in \eqref{eqn:final_2x2system} and solve the resulting $2\times 2$ linear system for $\delta_1,\delta_2$.
	\item Finally, let $\bm{\alpha}=\bm{\alpha}_0+\delta_1\mathbf{v}_1+\delta_2\mathbf{v}_2,$ and compute
	\begin{align*}
		\mathcal{Q}_{\omega}^{{L,[\nu]}}[f]=\sum_{n=0}^{\nu+1}\alpha_n e^{i\omega g(1)}-\sum_{n=0}^{\nu+1}(-1)^n\alpha_n e^{i\omega g(-1)}.
	\end{align*}
\end{enumerate}
\subsection{The inclusion of endpoint derivative values in the scalar case, $\mathbf{M=1,s\geq 1}$}\label{sec:including_endpoint_derivative_values}
Next we consider the case $s\geq 1, M=1$, still requiring that our integral is of the form given in Example~\ref{ex:exponential_oscillator}. In this case only a slight modification of the aforementioned algorithm is required in order to solve the collocation equations \eqref{eqn:general_levin_collocation_equations} in $\mathcal{O}((s+1)(\nu\log\nu+d^2\nu)+s^3)$ operations. Let us begin by introducing the projection $\mathsf{P}_{\nu+2}:\mathbb{C}^{\nu+2s+1}\rightarrow\mathbb{C}^{\nu+1}$
\begin{align}\label{eqn:definition_P_nu+2}
	\mathsf{P}_{\nu+2}(x_0,\dots,x_{\nu+1},x_{\nu+2},\dots,x_{\nu+2s+1})&=(x_0,\dots,x_{\nu+1}).
\end{align}
This time our basis is $\{\ChebyshevT_n\}_{n=0}^{\nu+2s+1}$ and the interpolation conditions arising in the Levin method are
\begin{align}\label{eqn:intermediate_conditions_levin1d}
	\sum_{n=0}^{\nu+2s+1} \alpha_n\mathcal{L}_\omega\ChebyshevT_{n}(c_m)=f(c_m), \quad m=0,\dots, \nu+1
\end{align}
where $\mathcal{L}_\omega=\dd/\dd x+i\omega g'(x)$, and, in addition, for $l=1,\dots,s,$ we have the conditions
\begin{align}\label{eqn:differential_conditions_levin1d}
	\sum_{n=0}^{\nu+2s+1} \alpha_n\left[\frac{\dd^l}{\dd x^l}\mathcal{L}_\omega\ChebyshevT_n\right]_{x=\pm1}=\left[\frac{\dd^lf}{\dd x^l}\right]_{x=\pm 1}.
\end{align}
To see how to solve these collocation equations efficiently for $\bm{\alpha}=(\alpha_0,\dots,\alpha_{\nu+2s+1})$ let us begin by defining the auxiliary vectors $\mathbf{h}^{{[j]}}\in\mathbb{C}^{\nu+2},\,j=1,\dots,s$ by
\begin{align*}
	h^{{[j]}}_m=\mathcal{L}_\omega\ChebyshevT_{\nu+j+1}(c_m),\quad m=0,\dots,\nu+1,j=1,\dots,2s.
\end{align*}
By Lemma~\ref{lem:banded_operators_Chebyshev_polynomials}, for $1\leq m\leq \nu$, the entries $h_m^{{[j]}}$ can be found efficiently using at most $\mathcal{O}(2d+1)$ operations each (since $\ChebyshevT_n'(c_m)=(n\ChebyshevT_{n-1}/2-n\ChebyshevT_{n+1}/2)/(1-c_m^2)$). The remaining entries with $m=0,\nu+1$ also require no more then $\mathcal{O}(2d+1)$ operations to compute since we have the following explicit expressions (proved in \cite[Eq.~(2.3)]{Gao2017a}):
\begin{align}\label{eqn:explicit_expressions_chebyshev_polys_at_pm1}
	\left[\frac{\dd^l}{\dd x^l}\ChebyshevT_n\right]_{x=\pm1}&=(\pm 1)^{n-l}\frac{2^l l! n(n+l-1)!}{(2l)!(n-l)!}, \quad \text{for\ }0\leq l\leq n \text{\ and\ }n+l\geq 1.
\end{align}
Note these expressions can be efficiently evaluated using the following recursive relation:
\begin{align*}
	\left[\frac{\dd^l}{\dd x^l}\ChebyshevT_n\right]_{x=\pm1}= \pm \frac{n^2-(l-1)^2}{2l-1} \left[\frac{\dd^{l-1}}{\dd x^{l-1}}\ChebyshevT_n\right]_{x=\pm1}
\end{align*}
Hence the overall cost of computing the vectors $\mathbf{h}^{{[j]}}, j=1,\dots,2s,$ is no more than $\mathcal{O}(ds(\nu+2))$ operations.

Once these vectors are computed we can solve $\mathsf{A}\bm{\beta}=\mathbf{f}$ and
\begin{align*}
	\mathsf{A}\bm{\beta}^{{[j]}}=-\mathbf{h}^{{[j]}},
\end{align*}
for each $j=1,\dots,2s$ in $\mathcal{O}(\nu\log\nu+d^2\nu)$ operations using the algorithm described in \S\ref{sec:initial_case_m-1_s-1}. Here we denoted by $\mathsf{A}$ the $(\nu+2)\times(\nu+2)$-matrix corresponding to the linear system \eqref{eqn:intermediate_conditions_levin1d}, i.e.
\begin{align*}
	\mathsf{A}_{mn}=\mathcal{L}_\omega \ChebyshevT_{n}(c_m),\quad m,n=0,\dots,\nu+1.
\end{align*}
We note that the conditions \eqref{eqn:intermediate_conditions_levin1d} are equivalent to
\begin{align*}
\sum_{n=0}^{\nu+1} \alpha_n\mathcal{L}_\omega\ChebyshevT_{n}(c_m)=f(c_m)-\sum_{j=1}^{2s}\alpha_{\nu+1+j}h^{{[j]}}_m,\quad m=0,\dots,\nu+1
\end{align*}
which can be written {in the form}
\begin{align*}
	\mathsf{A}\mathsf{P}_{\nu+2}\bm{\alpha}=\mathbf{f}-\sum_{j=1}^{2s}\alpha_{\nu+1+j}\mathbf{h}^{{[j]}}.
\end{align*}
Therefore (for sufficiently large $\omega$ we know that $\mathsf{A}$ is invertible)
\begin{align*}
	\mathsf{P}_{\nu+2}\bm{\alpha}=\bm{\beta}+\sum_{j=1}^{2s}\alpha_{\nu+1+j}\bm{\beta}^{{[j]}}.
\end{align*}
Thus we have determined the first $\nu+2$ coefficients $\alpha_0,\dots,\alpha_{\nu+1}$ in terms of the remaining $2s$ unknowns $\alpha_{\nu+2},\dots,\alpha_{\nu+2s+1}$. These latter coefficients can be found from the auxiliary $2s\times 2s$ linear system given by the conditions \eqref{eqn:differential_conditions_levin1d}, which simplify to
\begin{align}\label{eqn:remaining_sxs_linear_system_levin1d}
	\sum_{j=1}^{2s} \alpha_{\nu+1+j}\left(\frac{\dd^l}{\dd x^l}\mathcal{L}_\omega\ChebyshevT_{\nu+1+j}(\pm1)-\sum_{n=0}^{\nu+1} \beta^{{[j]}}_n\frac{\dd^l}{\dd x^l}\mathcal{L}_\omega\ChebyshevT_n(\pm1)\right)=f^{{[j]}}(\pm 1)-\sum_{n=0}^{\nu+1} \beta_n\frac{\dd^l}{\dd x^l}\mathcal{L}_\omega\ChebyshevT_n(\pm1),
\end{align}
for $l=1,\dots, s$. Note {that}, using \eqref{eqn:explicit_expressions_chebyshev_polys_at_pm1}, the computation of the coefficients in this linear system takes $\mathcal{O}(s^2)$ additional operations. The small linear system \eqref{eqn:remaining_sxs_linear_system_levin1d} is, in general, dense and hence takes $\mathcal{O}(s^3)$ operations to solve.
\subsubsection*{An algorithm for the efficient construction of the Levin--Clenshaw--Curtis method ($\mathbf{M=1,s\geq1}$)}
Thus we can summarise the extension of our algorithm for the efficient computation of the Levin method in the case $s\geq1$ as follows:
\begin{enumerate}
	\item Using the algorithm described in \S\ref{sec:initial_case_m-1_s-1} find $\bm{\beta},\bm{\beta}^{{[j]}}\in\mathbb{C}^{\nu+2}$ with
	\begin{align*}
		\mathsf{A}\bm{\beta}&=\mathbf{f},\quad		\mathsf{A}\bm{\beta}^{{[j]}}=-\mathbf{h}^{{[j]}},
	\end{align*}
where $h_m^{{[j]}}=\mathcal{L}_\omega \ChebyshevT_{\nu+1+j}(c_m)$ and $\mathsf{A}_{mn}=\mathcal{L}_\omega\ChebyshevT_n(c_m)$.
\item Express the first $\nu+2$ coordinates of $\bm{\alpha}\in\mathbb{C}^{\nu+2s+2}$ in terms of $\alpha_{\nu+2},\dots,\alpha_{\nu+2s+1}$,
\begin{align*}
	\mathsf{P}_{\nu+2}\bm{\alpha}=\bm{\beta}+\sum_{j=1}^{2s}\alpha_{\nu+1+j}\bm{\beta}^{{[j]}}.
\end{align*}
\item Solve the remaining $2s\times 2s$ auxiliary system \eqref{eqn:remaining_sxs_linear_system_levin1d} for $\alpha_{\nu+2},\dots,\alpha_{\nu+2s+1}$.
\item This solution then uniquely determines $\alpha_{0},\dots,\alpha_{\nu+2s+1}$ and we can compute the Levin approximation to $I_\omega[f]$ by
\begin{align*}
	\mathcal{Q}_{\omega}^{L,[\nu,s]}[f]=\sum_{n=0}^{\nu+2s+1}\alpha_n e^{i\omega g(1)}-\sum_{n=0}^{\nu+2s+1}(-1)^n\alpha_n e^{i\omega g(-1)}.
\end{align*}
\end{enumerate}
This takes overall $\mathcal{O}(s(\nu\log\nu+d^2\nu)+s^3)$ operations (assuming $s\geq 1,\nu\geq 2$).
\subsection{Vectorial oscillatory weight functions, $\mathbf{M\geq2,s=0}$}\label{sec:vectorial_weight_functions}
In this case we choose $M$ approximation spaces, each consisting of a Chebyshev polynomial basis, such that
\begin{align*}
	q_j(x)=\sum_{n=0}^{\nu+1}\alpha_{n}^{{[j]}}\ChebyshevT_{n}(x),\quad j=1,\dots,M.
\end{align*}
Our goal is then to solve the system of collocation equations \eqref{eqn:general_levin_collocation_equations} with $s=0$. Since the entries of $\mathsf{G}_\omega$ are rational functions in $x$ and without poles in $[-1,1]$ the Levin differential equation \eqref{eqn:general_Levin_differential_equation} is equivalent to
\begin{align*}
	r(x)\mathbf{u}'(x)+r(x)\mathsf{G}_\omega^\top(x)\mathbf{u}(x)=r(x)\mathbf{f}(x),
\end{align*}
where $r(x)$ is the least common multiple of all denominators in $\mathsf{G}_\omega$ such that each entry of $r(x)\mathsf{G}_\omega$ is a polynomial in $x$. Thus the collocation solution of \eqref{eqn:general_Levin_differential_equation} is equivalent to the following linear system
\begin{align}\label{eqn:collocation_equations_M>1}
	\mathsf{A}\bm{\alpha}=\mathbf{F},
\end{align}
where
\begin{align}\label{eqn:def_of_block_A}
	\mathsf{A}=\begin{pmatrix}
		\mathsf{A}^{{[1,1]}}&\cdots &\mathsf{A}^{{[1,M]}}\\
		\vdots&\ddots&\vdots\\
		\mathsf{A}^{{[M,1]}}&\cdots&\mathsf{A}^{{[M,M]}}
	\end{pmatrix},\quad \bm{\alpha}=\begin{pmatrix}
	\bm{\alpha}^{{[1]}}\\
	\vdots\\
	\bm{\alpha}^{{[M]}}
\end{pmatrix},\quad\mathbf{F}=\begin{pmatrix}
\mathbf{F}^{{[1]}}\\
\vdots\\
\mathbf{F}^{{[M]}}
\end{pmatrix},
\end{align}
and more specifically each block in this linear system is defined as
\begin{align*}
	\mathsf{A}^{{[i,j]}}_{mn}=\left[r\frac{\dd\ChebyshevT_n}{\dd x}+(r\mathsf{G}_\omega)_{ij}\ChebyshevT_n\right]_{x=c_m},\quad \mathbf{F}^{{[i]}}_m=r(c_m)f_i(c_m),\quad 1\leq i,j\leq M,0\leq m,n\leq \nu+1.
\end{align*}
Note {that} for consistency with the previous sections we shall index vectors in $\mathbb{C}^{M(\nu+2)}$ by $i\in\{0,1,\dots,\nu+1,\nu+2,\dots, (M-1)(\nu+2)+\nu+1\}$. Let $d_1=\deg r(x)$ denote the polynomial degree of $r(x)$, $d_2=\max_{1\leq i,j\leq M}\left(r(x)\mathsf{G}_\omega(x)\right)_{ij}{-1}$ be the polynomial degree of any entry of the matrix, and define $d=\max\{d_1,d_2\}$: this will determine the bandwidth of the banded matrix representation of the action of the modified Levin differential operator on our basis functions in this case. Let us consider the following differential operator
\begin{align*}
	\widetilde{\bm{\mathcal{L}}}_\omega:=(1-x^2)r(x)\frac{\dd}{\dd x}+(1-x^2)r(x)\mathsf{G}_\omega^\top.
\end{align*}
Then the action of $\widetilde{\bm{\mathcal{L}}}_\omega$ on a general vector of Chebyshev polynomial basis functions $(\ChebyshevT_{n_1},\dots,\ChebyshevT_{n_{M}})^\top$ is of the form
\begin{align*}
		\widetilde{\bm{\mathcal{L}}}_\omega(\ChebyshevT_{n_1},\dots,\ChebyshevT_{n_{M}})^\top(x)=\left(\sum_{j=1}^M\sum_{k=\max\{0,n_j-(d+1)\}}^{n_j+d+1}\mathsf{B}^{{[1,j]}}_{kn_j}\ChebyshevT_k(x),\dots,\sum_{j=1}^M\sum_{k=\max\{0,n_j-(d+1)\}}^{n_j+d+1}\mathsf{B}^{{[M,j]}}_{kn_j}\ChebyshevT_k(x)\right),
\end{align*}
Let us define in the spirit of \eqref{eqn:def_of_tilde{B}_1d_case}
\begin{align*}
			\widetilde{\mathsf{B}}^{{[i,j]}}_{nm}=\begin{cases}
		\mathsf{B}^{{[i,j]}}_{nm},&0\leq n<\nu-d, n=\nu+1,\\
		\mathsf{B}^{{[i,j]}}_{nm}+\mathsf{B}^{{[i,j]}}_{(2\nu+2-n)m},&\nu-d\leq n\leq \nu.
	\end{cases}
\end{align*}
By using $\ChebyshevT_{\nu+1+l}(c_m)=\ChebyshevT_{\nu+1-l}(c_m),\,0\leq l, m\leq \nu+1,$ (cf. \eqref{eqn:T_nu+1+l_equals_T_nu+1-l}) we can thus simplify the collocation equations for the modified operator $	\widetilde{\bm{\mathcal{L}}}_\omega$ by the $M(\nu+2)\times M(\nu+2)$-matrix $\widetilde{\mathsf{A}}$ which is given by
\begin{align*}
	\widetilde{\mathsf{A}}=\underbrace{\begin{pmatrix}
		\mathsf{C}&\mathsf{0}&\dots&\mathsf{0}\\
		\mathsf{0}&\mathsf{C}&\ddots&\vdots\\
		\vdots&\ddots&\ddots&\mathsf{0}\\
		\mathsf{0}&\cdots&\mathsf{0}&\mathsf{C}
	\end{pmatrix}}_{=:\mathsf{C}_M}\widetilde{\mathsf{B}},\quad\widetilde{\mathsf{B}}=\begin{pmatrix}
		\widetilde{\mathsf{B}}^{{[1,1]}}&\cdots &\widetilde{\mathsf{B}}^{{[1,M]}}\\
		\vdots&\ddots&\vdots\\
		\widetilde{\mathsf{B}}^{{[M,1]}}&\cdots&\widetilde{\mathsf{B}}^{{[M,M]}}
	\end{pmatrix},
\end{align*}
where, as we had in \S\ref{sec:initial_case_m-1_s-1}, $\mathsf{C}_{mk}=\ChebyshevT_k(c_m)$. We now define $\mathsf{P}_{M,\nu},$ the block-wise projection onto the middle $\nu$ coordinates, $\mathsf{P}_{M,\nu}:\mathbb{C}^{M(\nu+2)}\rightarrow \mathbb{C}^{M(\nu+2)},$ by
\begin{align*}
\mathsf{P}_{M,\nu}\begin{pmatrix}
	\mathbf{x}^{{[1]}}\\
	\vdots\\
	\mathbf{x}^{{[M]}}
\end{pmatrix}=\begin{pmatrix}
\mathsf{P}_{\nu}\mathbf{x}^{{[1]}}\\
\vdots\\
\mathsf{P}_{\nu}\mathbf{x}^{{[M]}}
\end{pmatrix},
\end{align*}
where $\mathsf{P}_\nu$ is defined in \S\ref{sec:initial_case_m-1_s-1}. Note that we can, by appropriately embedding $\mathbb{C}^{M\nu}$ into $\mathbb{C}^{M(\nu+2)}$, alternatively think of this projection as a surjective map onto $\mathbb{C}^{M\nu}$. Analogously to our treatment in the case when $M=1$ we shall consider the linear system
\begin{align*}
	\mathsf{P}_{M,\nu}\mathsf{A}\mathsf{P}_{M,\nu}\bm{\alpha}_0=\mathsf{P}_{M,\nu}\mathbf{F},
\end{align*}
which is equivalent to 
\begin{align}\label{eqn:projected_system_case_M>1}
	\mathsf{P}_{M,\nu}\widetilde{\mathsf{A}}\mathsf{P}_{M,\nu}\bm{\alpha}_0=\mathsf{P}_{M,\nu}\tilde{\mathbf{F}},
\end{align}
where analogously to the case $M=1$ the vector $\tilde{\mathbf{F}}$ is given by
\begin{align*}
	\tilde{\mathbf{F}}=\begin{pmatrix}
		\tilde{\mathbf{F}}^{{[1]}}\\
		\vdots\\
		\tilde{\mathbf{F}}^{{[M]}}
	\end{pmatrix},\quad \tilde{\mathbf{F}}^{{[i]}}_m=(1-c_m^2){\mathbf{F}}^{{[i]}}_m,\quad 1\leq i\leq M,0\leq m\leq \nu+1.
\end{align*}

Before attempting to solve \eqref{eqn:projected_system_case_M>1} let us prove that, for $\omega$ sufficiently large, this $(M\nu)\times(M\nu)$-system has a unique solution:
\begin{proposition}\label{prop:well-posedness_of_smaller_system_M>1}
	For $\omega$ sufficiently large, the system \eqref{eqn:projected_system_case_M>1} has a unique solution $\bm{\alpha}_0\in\mathbb{C}^{M(\nu+2)}$ with $\bm{\alpha}_0=\mathsf{P}_{M,\nu}\bm{\alpha}_0$.
\end{proposition}
Our proof will use the following result from \cite{Olver2007}:
\begin{lemma}[{Thm.~2.1 in \cite{Olver2007}}]\label{lem:invertability_of_sum_small_large_matrix}
	Suppose that $\mathsf{A} = \mathsf{X} + \mathsf{G}$ is a square matrix. If $\mathsf{X} = o(1)$ and $\mathsf{G}$ is invertible with $\mathsf{G}^{-1} = \mathcal{O}(1)$ (both as $\omega\rightarrow\infty$), then $\mathsf{A}$ is nonsingular when $\omega$ is sufficiently large and $\mathsf{A}^{-1}=\mathcal{O}(1)$ as $\omega\rightarrow\infty$.
\end{lemma}
\begin{proof}[Proof of Prop.~\ref{prop:well-posedness_of_smaller_system_M>1}]
We extend the proof of Prop.~\ref{prop:efficient_inversion_of_subsystem} to the case $M>1$: Firstly, we can write the system \eqref{eqn:projected_system_case_M>1} in the form
\begin{align*}
	\begin{pmatrix}\sum_{n=1}^{\nu}\alpha_{n}^{{[1]}}\ChebyshevT_{n}(c_m)\\
		\vdots\\
		\sum_{n=1}^{\nu}\alpha_{n}^{{[M]}}\ChebyshevT_{n}(c_m)
	\end{pmatrix}=\left[\left(\mathsf{G}_\omega^\top\right)^{-1}\mathbf{f}\right]_{x=c_m}+{o(1)},\quad m=1,\dots,\nu, \text{as}\ \omega\rightarrow\infty,
\end{align*}
i.e. $\mathsf{P}_{M,\nu}\mathsf{C}_M\mathsf{P}_{M,\nu}=\mathsf{P}_{M,\nu}\mathbf{g}$, some $\mathbf{g}\in\mathbb{C}^{M\nu}$. Thus it remains to show that the equation
\begin{align*}
	\mathsf{P}_{M,\nu}\mathsf{C}_M\mathbf{y}=\mathsf{P}_{M,\nu}\mathbf{z}
\end{align*}
has a unique solution $\mathbf{y}\in\mathbb{C}^{M(\nu+2)}$ with $\mathbf{y}=\mathsf{P}_{M,\nu}\mathbf{y}$ for any $\mathbf{z}\in\mathbb{C}^{M(\nu+2)}$. Indeed this just the block diagonal analogue of our proof of Prop.~\ref{prop:efficient_inversion_of_subsystem} and so the result follows by applying the same arguments block-wise and using Lemma~\ref{lem:invertability_of_sum_small_large_matrix}.
\end{proof}

Having proved well-posedness we can now think of the system \eqref{eqn:projected_system_case_M>1} as a square $(M\nu)\times (M\nu)$ linear system on $\mathbb{C}^{M\nu}$. This system can be solved efficiently in a spirit similar to the scalar case $M=1$:

\begin{proposition}\label{prop:fast_solution_linear_system_M>1}
	{Once} a solution exists, we can solve \eqref{eqn:projected_system_case_M>1} using $\mathcal{O}(M\nu\log\nu+M^3d^2\nu)$ operations, by a simple reordering followed by the block-wise application of discrete cosine transforms (DCT-I) and the solution of a banded matrix system of bandwidth at most $2M(d+4)-1$.
\end{proposition}
The efficient solution of \eqref{eqn:projected_system_case_M>1} can be performed by blockwise operations similar to the process introduced in the proof of Prop.~\ref{prop:efficient_inversion_of_subsystem}, with the following additional observation on the fast inversion of $\widetilde{\widetilde{\mathsf{B}}}:=\mathsf{P}_{M,\nu}\widetilde{\mathsf{B}}\mathsf{P}_{M,\nu}$. Note {that} each $(\nu+2)\times(\nu+2)$-block of the matrix $\widetilde{\widetilde{\mathsf{B}}}$ is a banded matrix of bandwidth at most $2d+1$. We can thus reorder rows and columns in such a way that inverting $\widetilde{\widetilde{\mathsf{B}}}$ is equivalent to the solution of a banded linear system with bandwidth $M(2d+5)$. A similar idea of reordering appears in the Hockney method \cite{hockney1965fast} (cf. also \cite{iserles2009first}) and was also kindly mentioned to the authors in personal communication with Webb \cite{webb2021}.

\begin{lemma}\label{lem:reordering_to_get_banded_system}
	The solution of $\widetilde{\widetilde{\mathsf{B}}}\mathbf{x}_1=\mathbf{y}_1$ is equivalent to the solution of $\mathsf{D}\mathbf{x}_2=\mathbf{y}_2$ where $\mathsf{D}$ is a banded matrix of bandwidth at most $2M(d+4)-1$.
\end{lemma}
\begin{proof}In this proof it is beneficial to index the entries of the matrix as $\widetilde{\widetilde{\mathsf{B}}}_{i,j}$ with indices taking values $1\leq i,j\leq M\nu$. Let us define the permutation $p:\{1,2,\dots,M\nu\}\rightarrow\{1,2,\dots,M\nu\}$ as follows: For a given $n\in\{1,2,\dots,M\nu\}$, let $l\in\{0,\dots, \nu-1\}$ be maximal such that $n=Ml+k$ for some $k\in\{1,\dots,M\}$ and define
	\begin{align*}
		p(n)=(k-1)\nu+l+1.
	\end{align*}
Then we construct $\mathsf{D},\mathbf{x}_2,\mathbf{y}_2$ as follows:
	\begin{align*}
		\mathsf{D}_{i,j}&=\widetilde{\widetilde{\mathsf{B}}}_{p(i),p(j)},\,
		x^{{[2]}}_j=x^{{[1]}}_{p(j)},\,
		y^{{[2]}}_j=y^{{[1]}}_{p(j)}.
	\end{align*}
Clearly the system $\widetilde{\widetilde{\mathsf{B}}}\mathbf{x}_1=\mathbf{y}_1$ is equivalent to ${\mathsf{D}}\mathbf{x}_2=\mathbf{y}_2$. Moreover $\mathsf{D}$ has the desired banded structure due to the following observation: Let $i,j\in \{1,2,\dots,M\nu\}$ and suppose that $|i-j|>(d+4)M$. Let us write $i,j$ uniquely in the form
\begin{align*}
	i&=l_1M+k_1\\
	j&=l_2M+k_2
\end{align*}
as introduced above. Then we have $D_{i,j}=\widetilde{\widetilde{\mathsf{B}}}_{p(i),p(j)}=\left(\mathsf{P}_{\nu}\widetilde{\mathsf{B}}^{{[k_1,k_2]}}\mathsf{P}_{\nu}\right)_{l_1+1,l_2+1}$. Now $|l_1-l_2|\geq \frac{1}{M}(|i-j|-|k_1-k_2|)> \frac{(d+4)M-M}{M}=d+3$ which implies that indeed $\left(\mathsf{P}_{\nu}\widetilde{\mathsf{B}}^{{[k_1,k_2]}}\mathsf{P}_{\nu}\right)_{l_1+1,l_2+1}=0$.
\end{proof}

\begin{example}
	To better visualise this process of reordering consider the following simple example where $M=2,\nu=3$ and each block $\widetilde{\mathsf{B}}^{{[n,m]}}$ is diagonal.
	\begin{align*}
	\tilde{B}=	\begin{pmatrix}
			b^{{[1,1]}}_1&0&0&b^{{[1,2]}}_1&0&0\\
			0&b^{{[1,1]}}_2&0&0&b^{{[1,2]}}_2&0\\
			0&0&b^{{[1,1]}}_3&0&0&b^{{[1,2]}}_3\\
			b^{{[2,1]}}_1&0&0&b^{{[2,2]}}_1&0&0\\
			0&b^{{[2,1]}}_2&0&0&b^{{[2,2]}}_2&0\\
			0&0&b^{{[2,1]}}_3&0&0&b^{{[2,2]}}_3\\
		\end{pmatrix}
	\end{align*}
Under the reordering $p$ as introduced above this matrix becomes
\begin{align*}
\mathsf{D}=\begin{pmatrix}
		b^{{[1,1]}}_1&b^{{[1,2]}}_1&0&0&0&0\\
		b^{{[2,1]}}_1&b^{{[2,2]}}_1&0&0&0&0\\
		0&0&b^{{[1,1]}}_2&b^{{[1,2]}}_2&0&0\\
		0&0&b^{{[2,1]}}_2&b^{{[2,2]}}_2&0&0\\
		0&0&0&0&b^{{[1,1]}}_3&b^{{[1,2]}}_3\\
		0&0&0&0&b^{{[2,1]}}_3&b^{{[2,2]}}_3\\
	\end{pmatrix}
\end{align*}
which is tridiagonal.
\end{example}

\begin{proof}[Proof of Prop.~\ref{prop:fast_solution_linear_system_M>1}]
	By block-wise extension of Claim~\ref{claim6:special_properties_of_DCTI_in_this_space} we have that \eqref{eqn:projected_system_case_M>1} is equivalent to the system
	\begin{align*}
	\mathsf{P}_{M,\nu}\widetilde{\mathsf{B}}\mathsf{P}_{M,\nu}\bm{\alpha}_{0}=\mathsf{P}_{M,\nu}	\mathsf{C}_{M}^{-1}\tilde{\mathbf{F}}
	\end{align*}
We note that $\mathsf{C}_{M}^{-1}$ can be written as $M$ blockwise DCT-I transforms, thus $\mathsf{C}_{M}^{-1}\tilde{\mathbf{F}}$ can be found in $\mathcal{O}(M \nu\log\nu)$ operations. By Lemma~\ref{lem:reordering_to_get_banded_system} this remaining system is, after suitable reordering, diagonal of bandwidth at most $2M(d+4)-1$ and size $M\nu$, thus it can be solved in $\mathcal{O}(M^3d^2\nu)$ operations.
\end{proof}
We have thus found $\bm{\alpha}_0\in \{\mathbf{x}\in\mathbb{C}^{M(\nu+2)}\,\big|\,x_{(k-1)(\nu+2)+l}=0, l=0,\nu+1,k=1,\dots, M\}$ with 
\begin{align*}
	\mathsf{P}_{M,\nu}\mathsf{A}\bm{\alpha}_0=\mathsf{P}_{M,\nu}\mathbf{F},
\end{align*}
and we can now proceed similarly to section \ref{sec:initial_case_m-1_s-1} and (noting that $\mathrm{Nullity}(\widetilde{\mathsf{A}})=\mathrm{Nullity}(\mathsf{A})+2M$) construct $2M$ linearly independent vectors $\mathbf{v}_{k,1},\mathbf{v}_{k,2}\in\mathbb{C}^{M(\nu+2)}, k=1,\dots, M,$ such that
\begin{align}\label{eqn:desired_null_solutions_M>1}
	\mathsf{P}_{M,\nu}\mathsf{A}\mathbf{v}_{k,1}=\mathsf{P}_{M,\nu}\mathsf{A}\mathbf{v}_{k,2}=\mathbf{0},\quad k=1,\dots,M.
\end{align}
We can construct those analogously to what has been done in section~\ref{sec:initial_case_m-1_s-1} for the case $M=1$: Let us denote the block-wise analogue of the endpoint unit vectors by
\begin{align*}
	(\mathbf{e}_{k,0})_{j}=\begin{cases}
		1,&\text{if}\ j=(k-1)(\nu+2),\\
		0,&\text{o.w.}
	\end{cases}, \quad 	(\mathbf{e}_{k,\nu+1})_{j}=\begin{cases}
	1,&\text{if}\ j=(k-1)(\nu+2)+\nu+1,\\
	0,&\text{o.w.}
\end{cases}.
\end{align*}
Then following Prop.~\ref{prop:fast_solution_linear_system_M>1} we can find  $\mathbf{v}_{k,1},\mathbf{v}_{k,2}\in \{\mathbf{x}\in\mathbb{C}^{M(\nu+2)}\,\big|\,x_{(k-1)(\nu+2)+l}=0, l=0,\nu+1,k=1,\dots, M\},  k=1,\dots, M,$ with
\begin{align}\label{eqn:linear_system_for_elements_of_kernel_M>1}
		\mathsf{P}_{M,\nu}{\mathsf{A}}\mathbf{v}_{k,1}=-\mathsf{P}_{M,\nu}\mathsf{A}\mathbf{e}_{k,0}, \quad \mathsf{P}_{M,\nu} \mathsf{A}\mathbf{v}_{k,2}=-\mathsf{P}_{M,\nu}\mathsf{A}\mathbf{e}_{k,\nu+1}.
\end{align}
Having found these vectors we can now write $\bm{\alpha}$ the solution of \eqref{eqn:collocation_equations_M>1} as
\begin{align*}
\bm{\alpha}=\bm{\alpha}_0+\sum_{k=1}^M\delta_1^{{[k]}}\mathbf{v}_{k,1}+\delta_2^{{[k]}}\mathbf{v}_{k,2}
\end{align*}
and determine the remaining $2M$ coefficients $\delta_{1}^{{[k]}},\delta_{2}^{{[k]}},k=1,\dots, M$ by solving the $2M\times 2M$ auxiliary system
\begin{align}\label{eqn:final_2Mx2M_system}
\begin{pmatrix}
	(\mathsf{A}\mathbf{v}_{1,1})_0&	(\mathsf{A}\mathbf{v}_{1,2})_0&\cdots &(\mathsf{A}\mathbf{v}_{M,2})_{0}\\
	(\mathsf{A}\mathbf{v}_{1,1})_{\nu+1}&(\mathsf{A}\mathbf{v}_{1,2})_{\nu+1}&\cdots &(\mathsf{A}\mathbf{v}_{M,2})_{\nu+1}\\
	(\mathsf{A}\mathbf{v}_{1,1})_{\nu+2}&(\mathsf{A}\mathbf{v}_{1,2})_{\nu+2}&\cdots&(\mathsf{A}\mathbf{v}_{M,2})_{\nu+2}\\
	\vdots&\vdots&\ddots&\vdots\\
	(\mathsf{A}\mathbf{v}_{1,1})_{M(\nu+2)-1}&(\mathsf{A}\mathbf{v}_{1,2})_{M(\nu+2)-1}&\cdots&(\mathsf{A}\mathbf{v}_{M,2})_{M(\nu+2)-1}
\end{pmatrix}\begin{pmatrix}
\delta_1^{{[1]}}\\
\delta_2^{{[1]}}\\
\delta_1^{{[2]}}\\
\vdots\\
\delta_2^{{[M]}}
\end{pmatrix}=\begin{pmatrix} F_0-(\mathsf{A}\bm{\alpha}_0)_0\\
 F_{\nu+1}-(\mathsf{A}\bm{\alpha}_0)_{\nu+1}\\
  F_{\nu+2}-(\mathsf{A}\bm{\alpha}_0)_{\nu+2}\\
  \vdots\\
   F_{M(\nu+2)-1}-(\mathsf{A}\bm{\alpha}_0)_{M(\nu+2)-1}
\end{pmatrix}.
\end{align}
\subsubsection*{An algorithm for the efficient construction of the Levin method ($\mathbf{M\geq2,s=0}$)}
In summary our algorithm for the efficient computation of the Levin method in the case $M\geq2, s=0$ thus takes the following steps:
\begin{enumerate}
	\item Solve \eqref{eqn:projected_system_case_M>1} using the method described in Prop.~\ref{prop:fast_solution_linear_system_M>1}, to find $\bm{\alpha}_0\in \{\mathbf{x}\in\mathbb{C}^{M(\nu+2)}\,\big|\,x_{(k-1)(\nu+2)+l+1}=0, l=0,\nu+1,k=1,\dots, M\}$ with
	\begin{align*}
		\mathsf{P}_{M,\nu}\widetilde{\mathsf{A}}\mathsf{P}_{M,\nu}\bm{\alpha}_0=\mathsf{P}_{M,\nu}\mathbf{F}.
	\end{align*}
	\item Solve \eqref{eqn:linear_system_for_elements_of_kernel_M>1} using the method described in Prop.~\ref{prop:fast_solution_linear_system_M>1}, to find $2M$ linearly independent vectors $\mathbf{v}_{k,1},\mathbf{v}_{k,2}, k=1,\dots M$ with 
	\begin{align*}
		\mathsf{P}_{M,\nu}\mathsf{A}\mathbf{v}_{k,j}=\mathbf{0},\quad j=1,2, k=1,\dots, M.
	\end{align*}
	\item Compute the coefficients in \eqref{eqn:final_2Mx2M_system} and solve the resulting $2M\times 2M$ linear system for $\delta_j^{{[k]}}, j=1,2,k=1,\dots, M$.
	\item Finally, let $\bm{\alpha}=\bm{\alpha}_0+\sum_{k=1}^M\delta_1^{{[k]}}\mathbf{v}_{k,1}+\delta_2^{{[k]}}\mathbf{v}_{k,2},$ and compute\begin{align*}
		\mathcal{Q}_{\omega}^{L,[\nu]}[f]&=\langle \mathbf{q}(1),\mathbf{w}(1)\rangle-\langle\mathbf{q}(-1),\mathbf{w}(-1)\rangle=\sum_{k=1}^{M}\left(q_k(1)w_k(1)-q_k(-1)w_k(-1)\right)\\
		&=\sum_{k=1}^M\sum_{n=0}^{\nu+1}\left(\alpha_n^{{[k]}}w_k(1)-(-1)^n\alpha^{{[k]}}_n w_k(-1)\right).
	\end{align*}
\end{enumerate}
This takes $\mathcal{O}(M\nu\log\nu+M^3d^2\nu)$ operations overall.

\subsection{Vectorial oscillatory weight functions with endpoint derivative values, $\mathbf{M\geq 2,s\geq 1}$}
Finally, the fast computation of the Levin method in the case $M\geq 2, s\geq1$, can be achieved through a similar extension of the above as presented in section~\ref{sec:including_endpoint_derivative_values}. In particular our $M$ polynomial basis spaces are each spanned by $\{\ChebyshevT_{n}\}_{n=0}^{\nu+2s+1}$ and we write the collocation solution as
\begin{align*}
	q_j(x)=\sum_{n=0}^{\nu+2s+1}\alpha_{n}^{{[j]}}\ChebyshevT_{n}(x),\quad j=1,\dots,M,
\end{align*}
for the coefficients $\alpha_{n}^{{[j]}}$ which are determined through the collocation conditions in Levin's method \eqref{eqn:general_levin_collocation_equations}. In the following the block-wise projection onto the first $\nu+2$ coordinates, $\mathsf{P}_{M,\nu+2}:\mathbb{C}^{M(\nu+2)}\rightarrow\mathbb{C}^{M(\nu+2)}$ {is} given by 
\begin{align*}
	\mathsf{P}_{M,\nu+2}\begin{pmatrix}
		\mathbf{x}^{{[1]}}\\
		\mathbf{x}^{{[2]}}\\
		\vdots\\
		\mathbf{x}^{{[M]}}
	\end{pmatrix}:=\begin{pmatrix}
	\mathsf{P}_{\nu+2}\mathbf{x}^{{[1]}}\\
	\mathsf{P}_{\nu+2}\mathbf{x}^{{[2]}}\\
	\vdots\\
	\mathsf{P}_{\nu+2}\mathbf{x}^{{[M]}}
	\end{pmatrix},\quad \mathbf{x}^{{[k]}}\in\mathbb{C}^{\nu+2s+2},\, k=1,\dots, M.
\end{align*}
We then construct the auxiliary vectors $\mathbf{h}^{{[k,j]}}\in\mathbb{C}^{\nu+2}$,
\begin{align}\label{eqn:definition_of_hkj}
	\mathbf{h}^{{[k,j]}}=\left(
		h_{1,0}^{{[k,j]}},
		h_{1,1}^{{[k,j]}},
		\dots,
		h_{1,\nu+1}^{{[k,j]}},
		h_{2,0}^{{[k,j]}},
		\dots,
		h_{M,\nu+1}^{{[k,j]}}\right)
\end{align}
with coordinates given by 
\begin{align*}
	{h}_{l,m}^{{[k,j]}}=\left(\left[\bm{\mathcal{L}}_\omega\left(\mathbf{e}_{k}\ChebyshevT_{\nu+j+1}(x)\right)\right]_{x=c_m}\right)_l,\quad m=0,\dots,\nu+1,j=1,\dots,2s, k=1,\dots, M,
\end{align*}
As in section~\ref{sec:including_endpoint_derivative_values}{,} the construction of these vectors requires no more than $\mathcal{O}(Mds(\nu+2))$ operations. In order to solve the Levin collocation problem efficiently we follow the recipe introduced in section~\ref{sec:including_endpoint_derivative_values} by separating the first $\nu+2$ coefficients in each component ($k=1,\dots, M$) of the problem from the remaining $2s$ to arrive at: (i) a large system that can be solved efficiently using the method described in section~\ref{sec:vectorial_weight_functions} and (ii) {another}, smaller system, that requires direct solution. In particular, these steps take the following form
\begin{enumerate}[(i)]
	\item Find $\bm{\beta}, \bm{\beta}^{{[k,j]}}, k$ which solve $\mathsf{A}\bm{\beta}=\mathbf{f}$ and $\mathsf{A}\bm{\beta}^{{[k,j]}}=-\mathbf{h}^{{[k,j]}}$ and note that 
	\begin{align*}
		\mathsf{A}\mathsf{P}_{M,\nu+2}\bm{\alpha}=\mathbf{f}-\sum_{k=1}^M\sum_{j=1}^{2s}\alpha_{nu+1+j}^{{[k]}}\mathbf{h}^{{[k,j]}}.
	\end{align*}
Thus we have
\begin{align*}
	\mathsf{P}_{M,\nu+2}\bm{\alpha}=\bm{\beta}+\sum_{k=1}^{M}\sum_{j=1}^{2s}\alpha_{\nu+1+j}^{{[k]}}\bm{\beta}^{{[k,j]}}{.}
\end{align*}
\item Consider the following $2Ms\times 2Ms$ linear system:
\begin{align}\label{eqn:remaining_auxiliary_system_vectorial_weights}
	\left[\frac{\dd^l}{\dd x^{l}}\bm{\mathcal{L}}_\omega\mathbf{q}\right]_{x=\pm1}&=\left[\frac{\dd^l}{\dd x^{l}}\mathbf{f}\right]_{x=\pm 1},\quad l=1,\dots,s,
\end{align}
to solve for the remaining unknowns $\alpha^{{[j]}}_{n}, j=1,\dots, M,n=\nu+2,\dots,\nu+2s+1$.
\end{enumerate}

\subsubsection*{An algorithm for the efficient construction of the Levin--Clenshaw--Curtis method ($\mathbf{M\geq2,s\geq1}$)}
Thus we can summarise the extension of our algorithm for the efficient computation of the Levin method in the case $s\geq1$ as follows:
\begin{enumerate}
	\item Using the algorithm described in \S\ref{sec:vectorial_weight_functions} find $\bm{\beta},\bm{\beta}^{{[k,j]}}\in\mathbb{C}^{M(\nu+2)}$ with
	\begin{align*}
		\mathsf{A}\bm{\beta}&=\mathbf{f},\quad		\mathsf{A}\bm{\beta}^{{[k,j]}}=-\mathbf{h}^{{[k,j]}},
	\end{align*}
	where $\mathbf{h}^{{[k,j]}}$ is defined in \eqref{eqn:definition_of_hkj} for $k=1,\dots, M$, $j=1,\dots,2s$ and $\mathsf{A}$ is as defined in \eqref{eqn:def_of_block_A}.
	\item Express the first $\nu+2$ coordinates of $\bm{\alpha}^{{[j]}}\in\mathbb{C}^{\nu+2s+2}, j=1,\dots, M$ in terms of $\alpha^{{[j]}}_{\nu+2},\dots,\alpha^{{[j]}}_{\nu+2s+1}$,
	\begin{align*}
	\mathsf{P}_{M,\nu+2}\bm{\alpha}=\bm{\beta}+\sum_{k=1}^{M}\sum_{j=1}^{2s}\alpha_{\nu+1+j}^{{[k]}}\bm{\beta}^{{[k,j]}}.
	\end{align*}
	\item Solve the remaining $2Ms\times 2Ms$ auxiliary system \eqref{eqn:remaining_auxiliary_system_vectorial_weights} for $\alpha^{{[j]}}_{\nu+2},\dots,\alpha^{{[j]}}_{\nu+2s+1}$, $j=1,\dots, M$.
	\item This solution then uniquely determines $\alpha^{{[j]}}_{0},\dots,\alpha^{{[j]}}_{\nu+2s+1}, j=1,\dots, M$ and we can compute the Levin approximation to $I_\omega[f]$ by 
	\begin{align*}
		\mathcal{Q}_{\omega}^{L,[\nu,s]}[f]&=\langle \mathbf{q}(1),\mathbf{w}(1)\rangle-\langle\mathbf{q}(-1),\mathbf{w}(-1)\rangle=\sum_{k=1}^{M}\left(q_k(1)w_k(1)-q_k(-1)w_k(-1)\right)\\
		&=\sum_{k=1}^M\sum_{n=0}^{\nu+2s+1}\left(\alpha_n^{{[k]}}w_k(1)-(-1)^n\alpha^{{[k]}}_n w_k(-1)\right).
	\end{align*}
\end{enumerate}
This takes $\mathcal{O}(s(M\nu\log\nu+M^3d^2\nu)+M^3s^3)$ operations overall.
\section{Numerical Experiments}\label{sec:numerical_experiments}
We now illustrate the efficient behaviour of {the} accelerated Levin--Clenshaw--Curtis method {with} two examples. {The} reference solution for the exact integral was computed using Clenshaw--Curtis quadrature on the full oscillatory integrand using $10^7$ points \revisionsA{and we compare our methods against the following reference methods where available:}
\begin{itemize}
	\item \revisionsA{The direct solution of the dense collocation system \eqref{eqn:chebyshev_interpolation_system_general} using Gaussian elimination using Matlab's ``backslash'' routine. We use direct solution of the full collocation system since we found iterative methods to struggle with the conditioning of this system in practice even for moderate values of $\nu$;}
	\item \revisionsA{The fast Levin method introduced by Olver \cite{olver2010fast};}
	\item \revisionsA{The} \verb|LevinRule| \revisionsA{in Mathematica, based on the adaptive Levin method introduced by Moylan \cite{moylan2008highly}.}
\end{itemize}
\subsection{Example 1: {Finite-length} Fourier transform}\label{sec:example1}
In the first example, we consider the case of \eqref{eqn:general1d_integral_levin} with $g(x)=x$, and $f(x)=x/(x^2+1/50)$, i.e. we wish to approximate the integral
\begin{align*}
	I^{(1)}_\omega[f]=\int_{-1}^1 \frac{x}{x^2+0.02}e^{i\omega x}\,\dd x.
\end{align*}
In this case the operator $\mathcal{L}_\omega$ takes the form $\mathcal{L}_\omega=\dd/\dd x+i\omega$ and one can use Lemma~\ref{lem:banded_operators_Chebyshev_polynomials} to check that $\widetilde{\mathcal{L}}_\omega$ has the following banded matrix representation in the sense of \eqref{eqn6:banded_matrix_representation_tildeL}:
\begin{align*}
	\mathsf{B}=\begin{pmatrix}
		-i\omega/2&-1/2&i\omega/4&0&\\
		0&-i\omega/4&-1&i\omega/4&0&\\
		i\omega/2&1/2&-i\omega/2&-3/2&i\omega/4&0\\
		0&i\omega/4&1&-i\omega/2&-2&i\omega/4&0&\\
		&&\ddots&\ddots&\ddots&\ddots&\ddots\\
	\end{pmatrix}.
\end{align*}

In Fig.~\ref{fig6:asymptotic_convergence_linear_levinCC} we see the error of the method, for a fixed number of collocation points $\nu$, as a function of the frequency $\omega$. It is apparent that the error remains uniformly bounded in $\omega$ and in fact decays as the frequency increases, meaning the method is even more accurate for larger frequencies while remaining convergent for small values of $\omega$.

\begin{figure}[h!]
	\centering
		\includegraphics[width=0.475\linewidth]{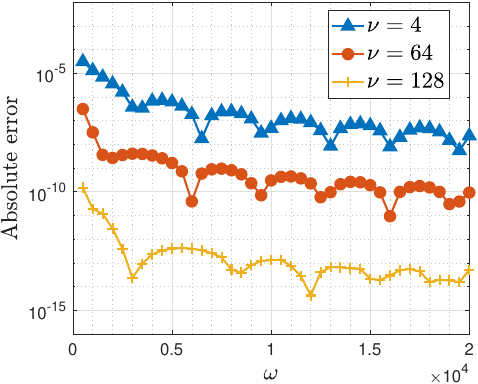}
	\caption{\revisionsA{The absolute error, $|\mathcal{Q}_\omega^{{L},[\nu]}[f]-I^{(1)}_\omega[f]|$, in the Levin--Clenshaw--Curtis method for $I_\omega^{(1)}[f]$ as a function of $\omega$ for fixed $\nu=4,64,128$.}}
	\label{fig6:asymptotic_convergence_linear_levinCC}
\end{figure}

In Fig.~\ref{fig6:nu_time_convergence_linear_levincc} we look at the error and timing of our method as we increase the number of quadrature points $\nu$ for a fixed value of the frequency $\omega$. In Fig.~\ref{fig6:nu_convergence_linear_levin_CC} we observe spectral convergence in $\nu$, which is to be expected since the amplitude function $f$ is analytic in an open complex neighbourhood of $[-1,1]$ \cite[Chapter~8]{trefethen2019approximation}. In Fig.~\ref{fig6:timing_nu_convergence_linear_levin_CC} we compare the time our method takes to compute the Levin approximation as described in \S\ref{sec:initial_case_m-1_s-1} against the \revisionsA{taken by reference methods described at the beginning of this section with the same number of collocation points}. \revisionsA{All} times were computed {using MATLAB} on a single core of an Apple M2 and the times shown in the graph correspond to the average time over 5 identical computations. As we predicted, the cost of our method appears to grow no faster than $\mathcal{O}(\nu\log\nu),$ \revisionsA{and, in particular, it appears to significantly outperform the reference methods.}

\revisionsA{We note from Figure~\ref{fig6:nu_convergence_linear_levin_CC} that no significant instabilities appear in our method even for large values of $\nu$. This is to be seen in the context of the full Levin collocation system $\mathsf{A}$ (cf. \eqref{eqn:chebyshev_interpolation_system_general}) which is known to be badly conditioned when $\omega>\nu$. In Figure~\ref{fig:condition_numbers2} we see the condition number of $\mathsf{A}$ compared against the condition number of the linear systems that require solution in our new approach, namely $\mathsf{P}_{\nu} \widetilde{\mathsf{B}}\mathsf{P}_{\nu}$ and the $2\times2$ system \eqref{eqn:final_2x2system}. The difference is striking and we observe that our methodology leads to reasonably well-conditioned large sparse systems $\mathsf{P}_{\nu} \widetilde{\mathsf{B}}\mathsf{P}_{\nu}$ and a remaining system \eqref{eqn:final_2x2system} which is badly conditioned but which is very small and thus does not need to rely on iterative solvers.}\newpage
\begin{figure}[h!]
	\centering
	\begin{subfigure}[h]{0.49\linewidth}
		\centering
		\includegraphics[width=0.95\linewidth]{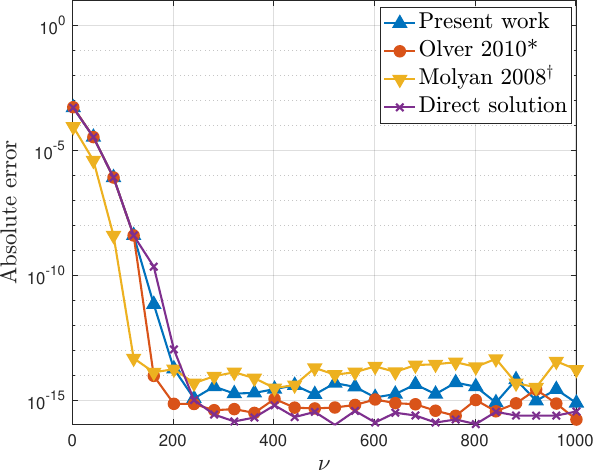}
		\caption{Absolute error, $|\mathcal{Q}_\omega^{{L},[\nu]}[f]-I^{(1)}_\omega[f]|$.}
		\label{fig6:nu_convergence_linear_levin_CC}
	\end{subfigure}%
	\begin{subfigure}[h]{0.49\linewidth}
		\centering
		\includegraphics[width=0.97\linewidth]{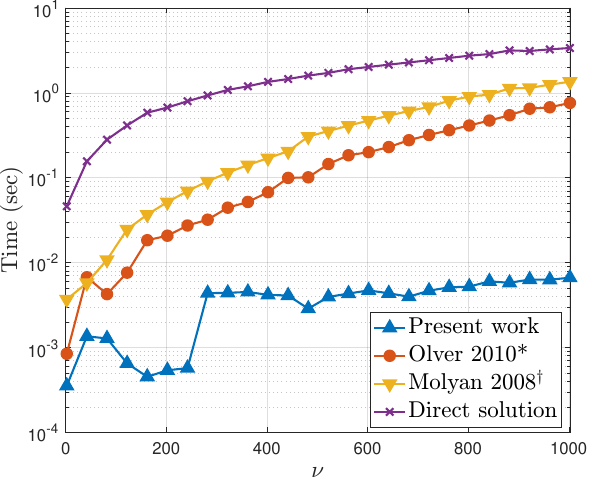}
		\caption{Average time to solve \eqref{eqn:general_discrete_interpolation_conditions}.}
		\label{fig6:timing_nu_convergence_linear_levin_CC}
	\end{subfigure}
	\caption{The error and timing of the Levin--Clenshaw--Curtis method as a function of the number of collocation points $\nu$ for fixed frequency $\omega=100$.}
	\label{fig6:nu_time_convergence_linear_levincc}
\end{figure}
\renewcommand{\thefootnote}{*} 
\footnotetext{\revisionsA{Note that Olver's method \cite{olver2010fast} has a cost of $\mathcal{O}(m^2\nu + m\nu\log\nu)$ where $m$ is the number of Krylov dimensions used in the corresponding GMRES solver. However, in practice we found that $m\sim \nu$ was necessary in order to achieve approximation accuracy comparable to the other methods, which leads to the observed computational cost.}}

\renewcommand{\thefootnote}{$\dagger$} 
\footnotetext{\revisionsA{We used the standard Mathematica implementation of Moylan's method \cite{moylan2008highly} but note that all the other methods are computed using Matlab code.}}

	\begin{figure}[h!]
		\centering
		\includegraphics[width=0.5\linewidth]{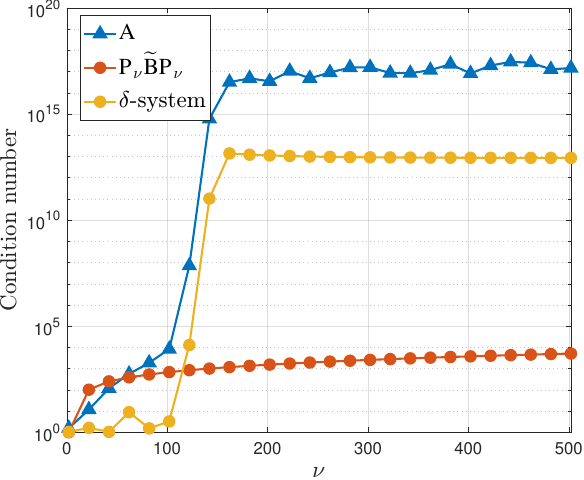}
		\caption{\revisionsB{The condition number of the full collocation system compared against the systems appearing in our present methodology.}}
		\label{fig:condition_numbers2}
	\end{figure}

\subsection{Example 2: Hankel transform}
We now consider the Hankel transform as introduced in Example~\ref{ex:bessel_weight_function} with $\gamma=1, a=2, f=x/(x^2+1/50)$, i.e. we seek to approximate the integral
\begin{align*}
	I^{(2)}_\omega[f]:=\int_{-1}^1\frac{x}{x^2+0.02}\mathrm{J}_1(\omega(x+2))\dd x.
\end{align*}
In the interest of brevity we do not print the form of the matrix $\widetilde{\mathsf{B}}$ in this case, but we note that it can easily be found using simple manipulations by hand or, alternatively, directly incorporated into the algorithmic implementation using sparse matrix manipulations starting from the two infinite matrices
\begin{align*}
	\mathsf{M}=\begin{pmatrix}
		0&1/2&0&\\
		1&0&1/2&0\\
		0&1/2&0&1/2&0\\
		&&\ddots&\ddots&\ddots
			\end{pmatrix}, \quad \mathsf{D}=\begin{pmatrix}
0&1/2&0\\
0&0&1&0\\
0&-1/2&0&3/2&0\\
&0&-1&0&2&0&\\
&&&\ddots&\ddots&\ddots&
	\end{pmatrix},
\end{align*}
where $\mathsf{M}$ represents the multiplication of the Chebyshev basis by $x$ and $\mathsf{D}$ represents the action of $(1-x^2)\dd/\dd x$ on the basis. In Fig.~\ref{fig6:asymptotic_convergence_Bessel_levinCC} we see again the error of the method for a fixed number of collocation points $\nu=4,64,128$, as a function of the frequency $\omega$. As expected we we see that again the error remains uniformly bounded in $\omega$ and in fact decays as the frequency increases, meaning the method is even more accurate for larger frequencies while remaining convergent for small values of $\omega$.

\begin{figure}[h!]
	\centering
	\includegraphics[width=0.475\linewidth]{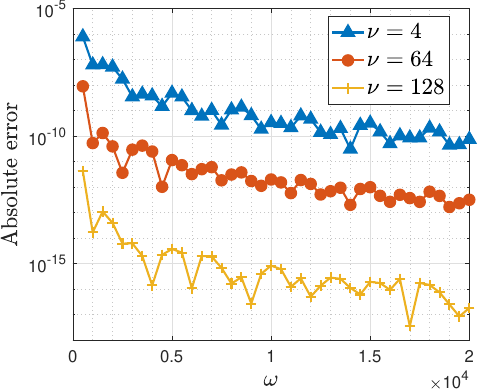}
	\caption{\revisionsA{The absolute error, $|\mathcal{Q}_\omega^{{L},[\nu]}[f]-I^{(2)}_\omega[f]|$, in the Levin--Clenshaw--Curtis method for $I_\omega^{(2)}[f]$ as a function of $\omega$ for fixed $\nu=4,64,128$.}}
	\label{fig6:asymptotic_convergence_Bessel_levinCC}
\end{figure}

Moreover, in Fig.~\ref{fig6:nu_time_convergence_Bessel_levincc} we see that also in the present case the computational cost of the method scales very favourably with the number of collocation points $\nu$. Again all times were computed {using MATLAB} on a single core of an Apple M2 and the times shown in the graph correspond to the average time over 5 identical computations. Of course, we achieve the known and usual convergence of the Levin method with respect to $\nu$ in this Clenshaw--Curtis setting uniformly in $\omega\geq 1$.

\begin{figure}[h!]
	\centering
	\begin{subfigure}[h]{0.49\linewidth}
		\centering
		\includegraphics[width=0.95\linewidth]{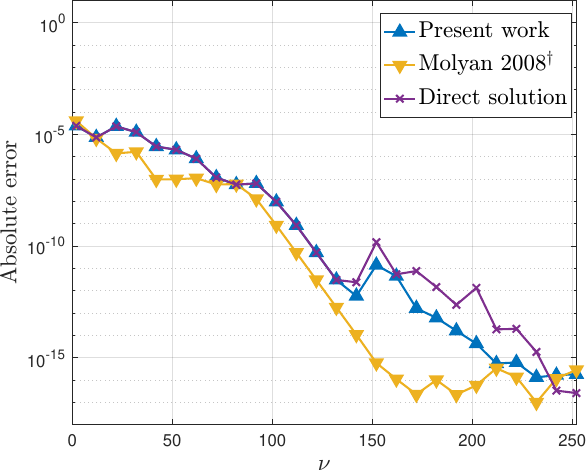}
		\caption{Absolute error, $|\mathcal{Q}_\omega^{{L},[\nu]}[f]-I^{(2)}_\omega[f]|$.}
		\label{fig6:nu_convergence_Bessel_levin_CC}
	\end{subfigure}%
	\begin{subfigure}[h]{0.49\linewidth}
		\centering
		\includegraphics[width=0.97\linewidth]{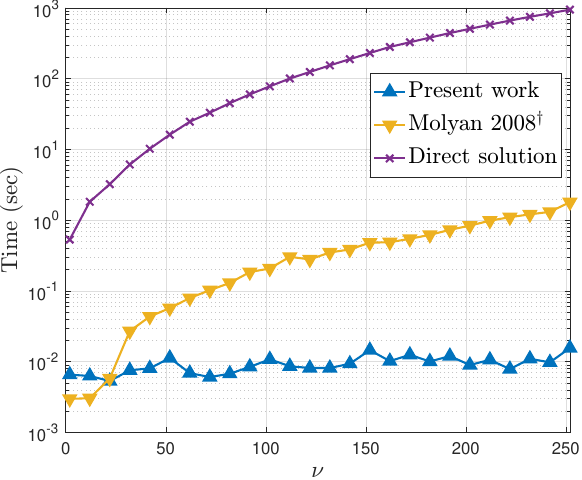}
		\caption{Average time to solve \eqref{eqn:general_discrete_interpolation_conditions}.}
		\label{fig6:timing_nu_convergence_Bessel_levin_CC}
	\end{subfigure}
	\caption{The error and timing of the Levin--Clenshaw--Curtis method for $I^{(2)}_\omega[f]$ as a function of the number of collocation points $\nu$ for fixed frequency $\omega=100$.}
	\label{fig6:nu_time_convergence_Bessel_levincc}
\end{figure}
\renewcommand{\thefootnote}{$\dagger$} 
\footnotetext{\revisionsA{We used the standard Mathematica implementation of Moylan's method \cite{moylan2008highly} but note that all the other methods are computed using Matlab code.}}
\section{Concluding remarks}\label{sec:conclusions}
We introduced an accelerated version of the Levin method for the efficient approximation of a wide class of highly oscillatory integrals. In particular we exploited the banded matrix representation of certain differential and multiplicative operators when acting on a Chebyshev polynomial basis in order to reduce the solution of the Levin collocation equations to essentially the inversion of a banded matrix system. We showed that this approach can be pursued in full generality of Levin's original formulation permitting also the use of vectorial weight functions and the inclusion of endpoint derivative values. Finally we demonstrated on two numerical examples that our approach indeed performs very well in practice.

Future research will focus on the extension of this idea to higher dimensional domains (cf. \cite{ashton2023}) and the use of alternative polynomial basis functions beyond Chebyshev polynomials (which allowed for fast inversions of the transformation matrix $\mathsf{C}$ above) for example based on recent work on fast polynomial transforms by Townsend et al. \cite{townsend2018fast}. {Another promising avenue for future work is to extend the fast Levin--Clenshaw--Curtis algorithm to integrals with stationary points, either through incorporation in an adaptive method (cf. \cite{chen2022adaptive}) or in the modified Levin--Olver method for integrals with stationary points (cf. \cite{olver2010fast}).} \revisionsB{In particular, the observations of the present manuscript can be used to show that the Levin--Clenshaw--Curtis collocation matrix can be related to a low-rank update of a sparse matrix, thus facilitating the use of designated solvers also for regularised versions of the corresponding linear systems (e.g. \cite{guttel2024sherman}) which form a central ingredient in the adaptive Levin method described by \cite{moylan2008highly,chen2022adaptive}.}
\section*{Acknowledgements}
The authors would like to thank Marcus Webb (University of Manchester) for several helpful discussions, in particular about the reordering of Hockney type which we used in \S\ref{sec:including_endpoint_derivative_values}. We also thank the anonymous reviewers for their valuable comments and suggestions, which have helped improve the manuscript. The authors gratefully acknowledge support from the UK Engineering and Physical Sciences Research Council (EPSRC) grant EP/L016516/1 for the University of Cambridge Centre for Doctoral Training, the Cambridge Centre for Analysis. GM would also like to thank the Isaac Newton Institute for Mathematical Sciences, Cambridge, for support and hospitality during the programme Mathematical theory and applications of multiple wave scattering, where work on this paper was undertaken. This work was supported by EPSRC grant EP/R014604/1.

\section*{Declarations}
\paragraph{Conflict of interest} The authors have no conflicts of interest to declare that are relevant to the content of this article.

\bibliographystyle{siam}     
\bibliography{biblio1}
\addcontentsline{toc}{section}{References}

\end{document}